\documentclass[11pt]{amsart}
\usepackage{geometry}     
\usepackage[parfill]{parskip}    
\usepackage{graphicx}
\usepackage{amssymb}
\usepackage{epstopdf}
\usepackage{xcolor}
\usepackage{fullpage}
\usepackage{tikz-cd}
\usepackage{mathrsfs}
\usepackage[all, 2cell]{xy}
\usepackage{longtable}

\usepackage{hyperref}
\usepackage{marginnote}
\usepackage[normalem]{ulem}

\usepackage[makeroom]{cancel}

\theoremstyle{plain}
\newtheorem{theorem}{Theorem}[section]
\newtheorem*{theorem*}{Theorem}

\newtheorem{lemma}[theorem]{Lemma}
\newtheorem{corollary}[theorem]{Corollary}
\newtheorem{proposition}[theorem]{Proposition}
\newtheorem{question}[theorem]{Question}

\theoremstyle{definition}
\newtheorem{example}[theorem]{Example}

\newtheorem{notation}[theorem]{Notation}
\newtheorem{definition}[theorem]{Definition}

\newtheorem{remark}[theorem]{Remark}
\newtheorem{remarks}[theorem]{Remarks}

\newcommand{\A}{\mathcal{A}}

\newcommand{\C}{\mathcal{C}}
\newcommand{\D}{\mathrm{D}}
\newcommand{\E}{\mathcal{E}}
\newcommand{\F}{\mathcal{F}}
\newcommand{\G}{\mathcal{G}}
\newcommand{\I}{\mathcal{I}}

\newcommand{\M}{\mathcal{M}}
\newcommand{\N}{\mathcal{N}}
\renewcommand{\L}{\mathcal{L}}

\renewcommand{\P}{\mathcal{P}}
\newcommand{\R}{\mathcal{R}}
\renewcommand{\S}{\mathcal{S}}
\newcommand{\T}{\mathcal{T}}
\newcommand{\U}{\mathcal{U}}
\newcommand{\V}{\mathcal{V}}
\newcommand{\W}{\mathcal{W}}
\newcommand{\X}{\mathcal{X}}
\newcommand{\Y}{\mathcal{Y}}
\newcommand{\Z}{\mathcal{Z}}

\renewcommand{\t}{\mathsf{t}}

\renewcommand{\u}{\mathsf{u}}

\newcommand{\fp}[1]{{#1}^{\mathsf{fp}}}
\renewcommand{\c}{\mathrm{c}}

\newcommand{\Mod}[1]{\mathrm{Mod}(#1)}

\renewcommand{\mod}[1]{\mathrm{mod}(#1)}

\newcommand{\Cogen}[1]{\mathrm{Cogen}(#1)}
\newcommand{\Der}[1]{\mathrm{D}(#1)}
\newcommand{\Db}[1]{\mathrm{D}^\mathrm{b}(\mathrm{mod}#1)}
\newcommand{\K}[1]{\mathrm{K}^{[0,1]}(\mathrm{Inj}#1)}
\newcommand{\Prod}[1]{\mathrm{Prod}(#1)}
\newcommand{\Kb}[1]{\mathrm{K}^{\mathrm{b}}(\mathrm{Inj}#1)}
\newcommand{\Inj}[1]{\mathrm{Inj}(#1)}
\newcommand{\Filt}[1]{\mathrm{Filt}(#1)}
\newcommand{\filt}[1]{\mathrm{filt}(#1)}

\newcommand{\Ctf}[1]{\mathcal{C}_{#1}}
\newcommand{\Ind}[1]{\mathrm{Ind}(#1)}

\newcommand{\kermor}[1]{\mathrm{ker}(#1)}
\newcommand{\cokermor}[1]{\mathrm{coker}(#1)}

\newcommand{\y}{\mathrm{y}}

\newcommand{\Ker}[1]{\mathrm{Ker}(#1)}
\newcommand{\Coker}[1]{\mathrm{Coker}(#1)}
\renewcommand{\Im}[1]{\mathrm{Im}(#1)}
\newcommand{\Hom}[1]{\mathrm{Hom}_{#1}}
\newcommand{\Ext}[1]{\mathrm{Ext}_{#1}}

\newcommand{\tstr}[1]{\mathbb{T}_{#1}}
\newcommand{\ststr}[1]{\mathbb{D}_{#1}}
\newcommand{\heart}[1]{\mathcal{H}_{#1}}
\renewcommand{\H}[1]{\mathrm{H}^0_{#1}}

\newcommand{\Zg}[1]{\mathbf{Zg}(#1)}
\newcommand{\open}[1]{\mathscr{O}(#1)}
\newcommand{\closed}[1]{\mathcal{N}_{#1}}
\newcommand{\ZgInt}[1]{\mathbf{Zg}^{[0,1]}}
\newcommand{\InjInd}[1]{\mathbf{Inj}(#1)}
\newcommand{\Sp}[1]{\mathbf{Sp}(#1)}
\newcommand{\Supp}[1]{\mathbf{Supp}(#1)}

\newcommand{\tors}[1]{\mathbf{tors}(#1)}
\newcommand{\ftors}[1]{{{\mathbf{f}\text{-}\mathbf{tors}}}(#1)}
\newcommand{\Cosilt}[1]{\mathbf{Cosilt}(#1)}
\newcommand{\CosiltZg}[1]{\mathbf{MaxRigid}(#1)}
\newcommand{\Cosiltpair}[1]{\mathbf{CosiltPair}(#1)}
\newcommand{\WideInt}[1]{\mathbf{WideInt}(#1)}
\newcommand{\Rigid}[1]{\mathbf{Rigid}(#1)}
\newcommand{\ClRigid}[1]{\mathbf{ClRigid}(#1)}

\newcommand{\Ann}[1]{\mathrm{Ann}(#1)}

\newcommand{\Tpair}[2]{(\mathcal{#1}, \mathcal{#2 })}
\newcommand{\tpair}[2]{(\mathbf{#1}, \mathbf{#2 })}
\newcommand{\Cpair}[2]{(\mathcal{#1}, \mathcal{#2})}

\newcommand{\chl}[1]{{\color{magenta}#1}} 

\title{Mutation of torsion pairs for finite-dimensional algebras}
\author{Lidia Angeleri H\"ugel, Rosanna Laking and Francesco Sentieri}
\thanks{Acknowledgments: The authors were partially supported by  the  Project 2022S97PMY \textit{Structures for Quivers, Algebras and Representations (SQUARE)}
funded by NextGenerationEU under NRRP, Call PRIN 2022 No. 104 of February 2, 2022 of Italian Ministry of University and Research, and by the project \textit{LAVIE - Large views of small phenomena: decompositions, localizations, and representation type}, FIS 00001706, funded by Program FIS2021 of Italian Ministry of University and Research. The first and second named authors are members of the network INdAM-G.N.S.A.G.A.}
\begin{document}
\def\theequation{\Alph{equation}}
\begin{abstract} We study the lattice $\tors{A}$ of torsion pairs in 
the category $\mod{A}$ of finitely generated modules over  an artinian ring $A$. We know from \cite{ALS1} that $\tors{A}$ is isomorphic to a lattice formed by certain closed sets, called maximal rigid,  in the Ziegler spectrum of the unbounded  derived category $\Der{A}$ of $A$. Moreover, the structure of this lattice is described by an operation on maximal rigid sets which encompasses (the dual of) silting mutation \cite{ALSV}. In this paper we provide an explicit description of this operation  and we discuss how it is reflected in the lattice $\tors{A}$. We establish a bijection between the wide intervals in $\tors{A}$ and the closed rigid sets in the Ziegler spectrum of  $\Der{A}$. Moreover, we show that the arrows in the Hasse quiver of $\tors{A}$ correspond to the closed rigid sets that are almost complete, or equivalently, that can be completed to a maximal rigid set in exactly two ways. Our results are most interesting in the case when $A$ is a finite dimensional algebra. In fact, we generalise results from \cite{AdachiIyamaReiten:14}, with an important difference: not every point in a maximal rigid set is mutable. We use the topology on the Ziegler spectrum to determine the mutable points. In the last section of the paper we illustrate our results by the example of a finite dimensional algebra arising from a triangulation of an annulus.  \end{abstract}
\maketitle
\setcounter{tocdepth}{1}
\tableofcontents

\section{Introduction}
The torsion pairs in the category $\mod{A}$ of finite dimensional modules over a finite dimensional algebra $A$ ordered by inclusion of the torsion class form a complete lattice $\tors{A}$ which encodes essential information on $A$. Adachi, Iyama and Reiten \cite{AdachiIyamaReiten:14} proposed to study this lattice from the viewpoint of silting theory. Indeed, they showed that the subposet $\ftors{A}$ formed by the functorially finite torsion pairs is isomorphic to the poset of 2-term silting complexes in $\mathrm{K}^{\mathrm{b}}(\mathrm{proj}(A))$, and its Hasse quiver reflects silting mutation. Their seminal paper opened up connections with a number of other areas of mathematics including cluster algebras, stability conditions and algebraic combinatorics.

A completely dual version of the Adachi-Iyama-Reiten picture (which one may call \emph{co}silting theory) can be obtained by applying the Nakayama equivalence $\nu \colon \mathrm{K}^{\mathrm{b}}(\mathrm{proj}(A)) \overset{\sim}{\to} \mathrm{K}^{\mathrm{b}}(\mathrm{inj}(A))$. In \cite{ALS1} we extend  this theory in two directions: we work on an arbitrary artinian ring rather than a finite dimensional algebra, and we generalise the poset isomorphism between $\ftors{A}$ and the poset of 2-term cosilting complexes. Indeed,  we are able to describe the whole of $\tors{A}$ by extending our view from the cosilting theory of the category $\mathrm{K}^{\mathrm{b}}(\mathrm{inj}(A))$ to that of the entire derived category $\Der{A}$ of $\Mod{A}$. More precisely, we show that torsion pairs in $\mod{A}$ correspond bijectively to sets of indecomposable pure-injective 2-term complexes in $\Der{A}$ which are \emph{rigid}, i.e.~there are no non-trivial extensions between any two elements, and are \emph{maximal} with this property (Definition~\ref{rig}). This generalises the Adachi-Iyama-Reiten cosilting theory in the sense that the maximal rigid sets that correspond to $\ftors{A}$ are exactly the sets of indecomposable summands of the corresponding 2-term cosilting complex in $\mathrm{K}^{\mathrm{b}}(\mathrm{inj}(A))$.

The aim of this paper is to explore the deep connections between the structure of the lattice $\tors{A}$ and mutations of the maximal rigid sets. The existence of a mutation operation on maximal rigid sets essentially comes down to reinterpreting the well-behaved notion of mutation of 2-term cosilting complexes in $\Der{A}$ introduced in \cite{ALSV} using the results of \cite{ALS1}. Having fixed some important notation and preliminaries in Section \ref{Sec: notation}, we dedicate Section \ref{Sec: mutation} to formalising this reinterpretation and laying out some first properties of mutation of maximal rigid sets.  

Given the bijection between $\tors{A}$ and the set of maximal rigid sets, mutation of maximal rigid sets can be interpreted as a mutation of torsion pairs. Indeed, it coincides with the one described by Asai and Pfeifer \cite{AsaiPfeifer}. Using this as our starting point we develop three aspects of the theory of mutations of maximal rigid sets that roughly correspond to Sections \ref{sec: crit+spec}, \ref{Sec: wide closed} and \ref{Sec: top char}. Finally, Section \ref{Sec: example} consists of an explicit description of irreducible mutations for cluster-tilted algebras of type $\tilde{A}$. The remainder of this introduction contains an overview of each of the final four sections of the paper.

 \subsection*{Critical and special objects (Section \ref{sec: crit+spec})}
 It was shown in \cite{ALSV} that two torsion pairs in the Hasse quiver of $\tors{A}$ are adjacent if and only if they are related by irreducible mutation. A more conceptual interpretation is achieved by considering Happel-Reiten-Smal\o\ tilts. Every torsion class is associated to a t-structure in $\Der{A}$; the t-structures arising in this way are called cosilting t-structures.
 It is shown in \cite{ALSV} that two torsion pairs are adjacent in the Hasse quiver of $\tors{A}$ if and only if the corresponding cosilting t-structures are related by a HRS-tilt at a  torsion pair generated by a simple object. 
 
In fact, the existence of this simple object is already implicit in \cite{BarnardCarrollZhu:19, DemonetIyamaReadingReitenThomas:23}, where it is shown that  every arrow $\mathbf{t}\to\mathbf{u}$ in the Hasse quiver of $\tors{A}$ is labeled by a brick $B$. When we pass to the hearts of the cosilting t-structures associated to $\mathbf{t}$ and $\mathbf{u}$, which are locally coherent Grothendieck categories, this brick label $B$ becomes a simple object. We thus have two simple objects, one in each heart, and the HRS-tilt  amounts to exchanging their injective envelopes. 
 
One of the aims of the present paper is to provide an explicit description of this process. Given a  torsion class $\mathbf{t}$ 
 associated to a maximal rigid set $\N_\mathbf{t}$ 
and a cosilting t-structure with heart $\heart{{\mathbf{t}}}$, there is a bijection between
the set of indecomposable injective objects in $\heart{{\mathbf{t}}}$ and  the closed set 
$\N_\mathbf{t}$. In Corollary~\ref{cor: inj env of simples}  we determine the points of $\N_\mathbf{t}$ that correspond to injective envelopes of simple objects in $\heart{{\mathbf{t}}}$.  To this end, we build on previous work in \cite{AHL} and obtain these points, which we call \emph{neg-isolated}, from an inspection of the left almost split maps in the torsion-free class $\F:=\mathbf{t}^{\perp_0}$.  

The simple objects in the heart  $\heart{\mathbf{t}}$  associated to a torsion class $\mathbf{t}$ are divided into two classes, depending on whether they are torsion or torsion-free with respect to the torsion pair in $\heart{\mathbf{t}}$ induced from $\mathbf{t}$. Similarly, their injective envelopes in $\heart{\mathbf{t}}$ and their counterparts in the maximal rigid set $\N_{\mathbf{t}}$, the  neg-isolated points,  
are divided into two classes, the \emph{critical} and the \emph{special} ones. 
An arrow  $\mathbf{t}\to\mathbf{u}$ in the Hasse quiver of $\tors{A}$  then encodes the existence of an exchange triangle which connects a special point   $\rho$ in $\N_\mathbf{t}$ with a critical point $\lambda$ in the maximal rigid set $\N_\mathbf{u}$ associated to $\mathbf{u}$, such that $\N_\mathbf{u}$ arises from $\N_\mathbf{t}$ by removing $\rho$ and replacing it with $\lambda$. This is exactly the situation in which $\N_{\mathbf{t}}$ and $\N_{\mathbf{u}}$ are related by irreducible mutation.

\subsection*{Wide intervals and closed rigid sets (Section \ref{Sec: wide closed})}

An interesting aspect of the bijection between $\tors{A}$ and the collection of maximal rigid sets is that each maximal rigid set $\N$ is a closed subset of the Ziegler spectrum of $\Der{A}$, a topological space whose points are parametrised by isomorphism classes of indecomposable pure-injective objects in $\Der{A}$. This opens up the question of whether this topology plays a role in the structure of $\tors{A}$.  In this section we show that certain closed sets do indeed control the wide intervals of $\tors{A}$.

Every arrow $\mathbf{t}\to\mathbf{u}$ in the Hasse quiver of $\tors{A}$ determines a  subcategory $\mathbf{t}\cap\mathbf{u}^{\perp_0}$ of $\mod{A}$ which consists of the modules which are finitely filtered by the brick label $B$. It is a wide subcategory, i.e.~an abelian subcategory closed under extensions. More generally, we consider the wide intervals $[\mathbf{u},\mathbf{t}]$ in $\tors{A}$, defined by the condition that $\W:=\mathbf{t}\cap\mathbf{u}^{\perp_0}$ is a wide subcategory of $\mod{A}$. In this case $\W$ is determined by a semibrick, a collection of Hom-orthogonal bricks which become simple objects when we pass to the hearts $\heart{{\mathbf{t}}}$ and $\heart{{\mathbf{u}}}$ associated to $\mathbf{t}$ and $\mathbf{u}$.
So, we can again exchange their injective envelopes and express this by exchange triangles  connecting  special points in  $\N_\mathbf{t}$ with critical points in $\N_\mathbf{u}$. This is exactly the situation in which $\N_{\mathbf{t}}$ and $\N_{\mathbf{u}}$ are related by mutation.

The closed rigid set  
$\N_\mathbf{t}\cap\N_\mathbf{u}$ consisting of the points that are not touched by this mutation process turns out to be an invariant of the interval  $[\mathbf{u},\mathbf{t}]$.

{\bf Theorem A.} (Theorem~\ref{thm: wide closed bijection} and Corollary~\ref{Cor: completions})
\emph{The assignment $[\mathbf{u}, \mathbf{t}]\mapsto\N_\mathbf{t}\cap\N_\mathbf{u}$ defines a bijection between the wide intervals in $\tors{A}$ and the closed rigid sets in the Ziegler spectrum of $\Der{A}$. The inverse map assigns to a closed rigid set $\M$  the wide  interval  $[\mathbf{u}_\M,\mathbf{t}_\M]$ in $\tors{A}$ formed by the torsion classes whose associated maximal rigid set  contains  $\M$.}

\subsection*{Topological characterisation of irreducible mutations (Section \ref{Sec: top char})}
In Section \ref{Sec: top char} we have two aims: first we wish to identify the closed rigid sets $\M$ for which the interval $[\mathbf{u}_\M,\mathbf{t}_\M]$ is minimal, i.e., it corresponds to an irreducible mutation, and secondly we wish to use the induced Ziegler topology on a maximal rigid set to identify which elements are swapped in some irreducible mutation.

Addressing the first aim, we find that the closed rigid sets $\M$ for which the interval $[\mathbf{u}_\M,\mathbf{t}_\M]$ is minimal are those that are `almost maximal'. We are able to show that these are also the closed sets for which there exist exactly two ways to complete them to a maximal rigid set. This generalises the (dual of the) result by Adachi, Iyama and Reiten \cite[Cor.~3.8]{AdachiIyamaReiten:14} showing that every almost complete 2-term silting complex has exactly two complements.

{\bf Theorem B.}(Theorem~\ref{Thm: topological char})
\emph{The following statements are equivalent for a closed rigid set $\M$.\\
(1) There is an arrow $\mathbf{t}_\M\to\mathbf{u}_\M$ in the Hasse quiver of $\tors{A}$.\\
(2) There are exactly two ways to complete $\M$ to a maximal rigid set.\\
(3) $\M$ is not maximal but  every closed rigid set  properly containing $\M$ is maximal.}

In contrast to silting mutation, cosilting mutation is not always possible. We say that a point $\mu$ in a maximal rigid set $\N$ is \emph{mutable} if there exists a maximal rigid set 
$\N'$ that arises from $\N$ by removing $\mu$ and replacing it with a point $\nu$    related by mutation with $\mu$. In general, a maximal rigid set  can contain points that are not mutable. For an example, we refer to Proposition \ref{prop: annulus mutable points}.
 
  From the discussion above, we know that the mutable points of $\N$ correspond to injective envelopes of simple objects  arising from brick labels, and the latter are precisely  the finitely presented simple objects in the heart associated to $\N$. In other words, the mutable points are  neg-isolated   points in $\N$ satisfying an additional finiteness condition that is not evident in $\N$.  The following result gives  an intrinsic description, showing that mutability of a point $\mu$ amounts to the fact that the set $\N \setminus \{\mu\}$ can be completed to a maximal rigid set in exactly two ways.  

{\bf Corollary C.} (Corollary~\ref{cor: mutable means exactly two completions})
\emph{Let $\N$ be a maximal rigid set. A point  $\mu\in\N$ is mutable if and only if there exists a unique  maximal rigid set $\N'\not=\N$ containing  $\N \setminus \{\mu\}$. In this case, $\mu$ is an isolated point of $\N$ in the Ziegler subspace topology on $\N$.}

This result raises the question whether the mutable points are precisely the  isolated points in the Ziegler topology. In Theorem~\ref{thm: iso prop} we give a positive answer under a  topological assumption on the   functor category associated to $A$ known as isolation condition. This assumption is verified for many important classes of finite dimensional algebras, in fact it holds whenever there are no  superdecomposable pure-injective $A$-modules.

\subsection*{Example: cluster-tilted algebras of type $\tilde{A}$ (Section \ref{Sec: example})}
In the final section of the paper we consider a family of finite-dimensional algebras over an algebraically closed field $\mathbb{K}$ that arise from triangulations of an annulus. These algebras form a special subfamily of the Jacobian algebras of the quivers with potential associated to marked surfaces introduced by Labardini-Fragoso \cite{LF} and have been shown in \cite{ABCP} to coincide with the family of cluster-tilted algebras of type $\tilde{A}$, i.e. those arising as the endomorphism ring of a cluster-tilting object in the cluster category of an acyclic quiver of type $\tilde{A}$.

Building on the classification of the cosilting modules over these algebras found in \cite{BaurLaking:22}, we give a classification of the maximal rigid sets in terms of asymptotic triangulations of the annulus in the sense of Baur and Dupont \cite{BaurDupont}. Using the theory developed in Section \ref{Sec: top char}, we show in Proposition \ref{prop: annulus mutable points} that the irreducible mutations of the maximal rigid sets correspond either to swapping the so-called Pr\"ufer modules and adic modules or to the combinatorial mutation operation described in \cite{BaurDupont}.

\section{Notation and preliminaries}\label{Sec: notation}

In this section we give an overview of the notions and results we will need in the article. After 
presenting our setup and listing the notation used in the paper, we review some basic results linking torsion pairs with cosilting objects and closed sets in  Ziegler spectra.

\subsection{The setup}

Throughout the paper, $A$ will be a left artinian ring. All subcategories will be strict and full. \mbox{We will fix the following notation for the various categories we will be working in.}
\begin{tabular}{l p{14cm}}
 $ \Mod{A} $ & The category of (all left) $ A$-modules. \\

$ \mod{A} $ & The subcategory of  finitely presented $ A$-modules.\\

 $ \Der{A} $ & The unbounded derived category of $ \Mod{A} $.\\

 $\Kb{A}$ &The subcategory of $\Der{A}$ consisting of objects isomorphic to complexes  of injective modules that are non-zero in only finitely many degrees.\\

$\K{A}$ & The subcategory of $\Der{A}$ consisting  of objects isomorphic to complexes of injective $A$-modules that are non-zero in degrees $0$ and $1$ only. Objects in this subcategory will sometimes be referred to as \textbf{two-term} complexes or two-term complexes of injective modules.
\end{tabular}

We regard modules as stalk complexes in $ \Der{A} $ concentrated in degree zero. Since every object $X$ in $\K{A}$ is determined by a morphism $\mu \colon I^0 \to I^1$ (the 0th differential) up to homotopy, we will often denote $X$ by $\mu$ and treat $\mu$ as a morphism in $\Mod{A}$ whenever it is convenient.

\subsection{Notation}
Given a class of objects $ \X$ in $\Der{A}$ (or in $\Mod{A}$), and a set $I \subseteq \mathbb{Z}$, we fix the following notation. We will use the notation $>0$ for the interval $\{i\in\mathbb{Z}\mid i>0\}$ and similarly for $\leq 0$.  We will write $\X^{\perp_i}$ when $I = \{i\}$.

\begin{tabular}{l p{14.5cm}}
$ \Ind{\X} $ & The class of isomorphism classes of indecomposable objects in $\X$.\\
$\Prod{\X} $ & The subcategory of $ \Der{A} $ (or of $\Mod{A}$, respectively) formed by the objects isomorphic to direct summands of set-indexed products of objects in $ \X $.\\
$\X^{\perp_I}$ & The subcategory of $ \Der{A} $ (respectively of $\Mod{A}$) consisting of the objects $ Y $ with $ \Hom{\Der{A}}(X, Y[i]) = 0 $ (respectively, $\Ext{A}^i (X,Y)=0$) for all $ X \in \mathcal{X} $ and all $i\in I$.   \\
${}^{\perp_I}\X$ & The subcategory of $ \Der{A} $ (respectively of $\Mod{A}$) consisting of the objects $ Y $ with $ \Hom{\Der{A}}(Y, X[i]) = 0 $ (respectively, $\Ext{A}^i (Y,X)=0$) for all $ X \in \mathcal{X} $ and all $i\in I$.   \\
$\Ctf{\X}$ & The subcategory of $\Mod{A}$ given by $\{ M \in \Mod{A} \mid \Hom{\Der{A}}(M, \X [1]) = 0 \}$.  Note that, when $\X\subseteq \K{A}$, we have that $\Ctf{\X}$ coincides with the subcategory given by $\{M\in \Mod{A}\mid \Hom{A}(M, \mu) \text{ is surjective for all } \mu\in\X\}$. If $\X=\{\mu\}$ we just write $\Ctf{\mu}$.
\end{tabular}

For a class of objects $\X$ in an abelian category $\A$ with products, we write the following. 

\begin{tabular}{l p{14.2cm}}  $\Cogen\X$ & The subcategory of $\A$ consisting of all objects isomorphic to subobjects of products of objects in $\X$.\\
$\Filt\X$ & The subcategory of all  objects $M$ which admit an ascending chain $(M_\lambda, \lambda \leq \mu)$ of subobjects indexed over an ordinal number $\mu$ where $M_0 = 0$,  all consecutive factors $M_{\lambda+1}/M_\lambda$ with $\lambda<\mu$ belong to $\X$, and  $M=\bigcup _{\lambda \le \mu} M_\lambda$. The class of objects with a finite filtration of this form is denoted by $\filt{\X}$.\\
$\varinjlim\X$ & The subcategory of $\A$ consisting of all objects obtained as colimits in $\A$ of direct systems in $\X$.
\end{tabular}

We will use the following notation for the various incarnations of the Ziegler spectrum.  
The interested reader can find some basics on the theory of purity in module categories and derived categories in \cite[Sec.~2.4 and App.~A]{ALS1} and the references therein. 

\begin{longtable}{l p{13.7cm}}
$\Sp{\heart{}}$ & The spectrum of a locally coherent Grothendieck category $\heart{}$.  The points of $\Sp{\heart{}}$ are the isomorphism classes of indecomposable injective objects in $\heart{}$, and the sets
	\[\open{C} = \{E \in \Sp{\heart{}} \mid \Hom{\G}(C, E) \neq 0 \}\] where $C$ is a finitely presented object of $\heart{}$ form a basis of open sets for the topology.  We will write $\InjInd{A}$ for special case of $\Sp{\Mod{A}}$.\\

$\Zg{A}$ & The Ziegler spectrum of a ring $A$.  The points of $\Zg{A}$ are the isomorphism classes of indecomposable pure-injective objects in $\Mod{A}$. The closed sets are the sets formed by the indecomposable pure-injective modules in definable subcategories. Here a subcategory of $\Mod{A}$ is said to be definable if it is closed under direct products,  pure submodules and pure epimorphic images. \\

$\Zg{\Der{A}}$ & The Ziegler spectrum of the derived category $\Der{A}$. It is defined in a similar way to $\Zg{A}$ using the notion of purity in triangulated categories.\\

$\ZgInt{A}$ & The subset of $\Zg{\Der{A}}$ consisting of isomorphism classes of complexes contained in $\K{A}$ with the subspace topology. This topological space is denoted by $\mathbf{Zg}^{[0,1]}(D(A))$ in \cite{ALS1}.
\end{longtable}

A crucial manoeuvre that we will need to make many times throughout the paper, is to move between $\K{A}$ and $\Mod{A}$ using the following operations.
\begin{itemize}
\item For a two-term complex $\mu \in \K{A}$, we can consider its zeroth cohomology $\H{}(\mu)$.  If $\N$ is a full subcategory of $\K{A}$, we will denote by $\H{}(\N)$ the subcategory of $\Mod{A}$ consisting of non-zero modules of the form $\H{}(\mu)$ for $\mu\in\N$.
\item For a module $M\in \Mod{A}$, we denote by $\mu_M$ the two-term complex given by the first two terms of the minimal injective resolution of $M$.
\end{itemize}

\subsection{The cosilting bijections}

In this subsection we review some bijections linking the poset of torsion classes $\tors{A}$ with posets formed by objects in the Ziegler spectrum. Let us first define the underlying sets of these posets. 

\begin{definition}\label{rig}
We say that an object $\mu$ in $\Der{A}$ is \textbf{rigid} if $\Hom{\Der{A}}(\mu, \mu[1]) = 0$. 
Furthermore, a set $\M$ of indecomposable objects in $\Der{A}$ is \textbf{rigid} if $\Hom{\Der{A}}(\mu, \nu[1]) = 0$ for all $\mu, \nu \in \M$.  Let $\Rigid{A}$ be the set of all rigid subsets of $\ZgInt{A}$ and let $\CosiltZg{A}$ denote the set of maximal elements of $\Rigid{A}$ with respect to inclusion.
\end{definition}

We collect some useful observations around the notion of rigidity.
\begin{lemma}\label{Csigma}\cite[Rem.~2.15, Lem.~3.5, 3.7 and 4.6]{ALS1} Let  $\sigma,\mu \in \K{A}$  and $M,N\in\Mod{A}$. 
\begin{enumerate}
\item $\Hom{\Der A}(\sigma, \mu[1]) = 0$
   if and only $\H{}(\sigma)\in\C_{\mu}$.
  \item $\Hom{\Der A}(\mu_M, \mu_N[1]) = 0$
if and only if  all 
 submodules of $M$ are contained in ${}^{\perp_1}N$. 
 
 In case $N$ is pure-injective, $\Hom{\Der A}(\mu_M, \mu_N[1])= 0$
   if and only all finitely generated
 submodules of $M$ are contained in ${}^{\perp_1}N$.   
  \item If  the module $\H{}(\sigma)$ is pure-injective  (in particular, this holds when the complex $\sigma$ is pure-injective), then $\C_{\sigma}$ is a  torsion-free class that is closed under directed colimits. 
  \item If $N$ is pure-injective, then $\Hom{\Der A}(\mu_M, \mu_N[1]) = 0$
if and only if   $\Cogen{M}\subseteq{}^{\perp_1}N$.
   \item If $M\in\Zg{A}$ and $\mu_M$ is rigid, then $\mu_M\in\ZgInt{A}$.  
   \end{enumerate}
\end{lemma}
\begin{proof} For (1) and (2) we refer to  \cite[Rem.~2.15]{ALS1}, and  (3) is shown in  \cite[Lem.~3.5]{ALS1}. 

(4) When {$N$} is pure-injective, $\C_{\mu_N}$ is closed under products and submodules by (3), and so the claim follows from (1) and (2), as well as the observation that $\C_{\mu_N} \subseteq {}^{\perp_1}N$.

(5)   We know from \cite[Lem.~3.7]{ALS1} that $\mu_M$ is  a direct summand of a pure-injective (2-term cosilting) complex. Thus $\mu_M$ is pure-injective, and by  \cite[Lem.~4.6]{ALS1} it is indecomposable.\end{proof}

Next, we  turn to the ``shadow''  of maximal rigid sets in the module category.

\begin{definition} \label{rpair} A \textbf{cosilting pair} 
 is a pair $(\Z,\I)$ given by a subset $\Z$ of $\Zg{A}$ and a subset  $\I$  of $\InjInd{A}$ such that
 \begin{enumerate} 
\item[(i)] The finitely generated submodules of $X$ are contained in ${}^{\perp_1}Y$ for all $X,Y\in \Z$;
\item[(ii)] $\Hom{A}(\Z,\I)=0$;
\item[(iii)]  Any pair $(\Z',\I')$  with (i) and (ii) such that  $\Z\subseteq\Z'$ and $\I\subseteq\I'$ satisfies $(\Z,\I)=(\Z',\I')$. 
\end{enumerate}
We denote the set of all cosilting pairs in $\Mod{A}$ by $\Cosiltpair{A}$.
\end{definition}

We are now ready to introduce orderings on these objects. 

\begin{definition} We define the following posets for a left artinian ring $A$.

\begin{tabular}{l p{14cm}}
$\tors{A}$ & The poset of torsion pairs $\tpair{u}{v}$ in $\mod{A}$ with $\tpair{u}{v}\leq\tpair{t}{f}$  when $\mathbf{u} \subseteq \mathbf{t}$.  \\

$\Cosilt{A}$ & The poset  of torsion pairs  $\Tpair{U}{V}$ of finite type in $\Mod{A}$ with $ \Tpair{U}{V}\leq \Tpair{T}{F}$  when $\mathcal{U} \subseteq \mathcal{T}$. Here we say that $\Tpair{U}{V}$ has {\bf finite type} if  the torsion-free class $\V$ is closed under directed colimits.\\

$\Cosiltpair{A}$ & The poset of cosilting pairs with $(\Z', \I')\leq (\Z, \I)$ when $\Cogen{\Z} \subseteq \Cogen{\Z'}$.\\

$\CosiltZg{A}$ & The poset of maximal rigid sets with $\M \leq\N$  when $\Hom{\Der{A}}(\N, \M[1])=0$.
\end{tabular}
\end{definition}

\begin{notation}
When the name of the torsion-free class is unimportant, we will often simplify the notation for elements of $\tors{A}$ by writing $\mathbf{u} \leq \mathbf{t}$ instead of $\tpair{u}{v}\leq\tpair{t}{f}$. 
\end{notation}

One of the main results of \cite{ALS1} is that there are order-preserving bijections between these posets. The poset $\tors{A}$ is known to be a complete lattice and so it follows that all four posets are complete lattices. We summarise those bijections in the following theorem.

\begin{theorem}\label{bij_posets}\cite[Cor.~3.11 and Thm.~4.3]{ALS1}
The following are order-preserving bijections for a left artinian ring $A$.
\begin{enumerate}
\item The assignment 
\[   \tors{A} \to \Cosilt{A}\]
given by $ \tpair{t}{f}\mapsto (\varinjlim \mathbf{t}, \varinjlim \mathbf{f})$, 
with inverse assignment  $\Tpair{T}{F}\mapsto (\T \cap \mod{A}, \F \cap \mod{A})$.

\item The assignment \[   \Cosilt{A} \to \Cosiltpair{A}\] defined by $ \Tpair{T}{F}\mapsto  (\Z,\I)$ where $\Z:=\Ind{\F\cap\F^{\perp_1}}$ and $\I:= \Z^{\perp_0} \cap \Inj{A}$,
with inverse assignment given by  $(\Z, \I)\mapsto ({}^{\perp_0}\Z, \Cogen{\Z})$.

\item The assignment \[ \Cosiltpair{A} \to \CosiltZg{A}\]
defined by $(\Z, \I)\mapsto\N:=\{\mu_M \mid M \in \Z\} \cup \{I[-1] \mid I\in\I\}$ with inverse assignment  given by $\N\mapsto (\H{}(\N), \I)$ where $\I := \{I\in \InjInd{A}\mid I[-1] \in \N\}$.
\end{enumerate}
\end{theorem}

\begin{notation}\label{fromt}
We will often want to consider elements of $\tors{A}, \Cosilt{A}, \Cosiltpair{A}$ and $\CosiltZg{A}$ that correspond to each other under the bijections.  To ease notation, we will always fix $\tpair{t}{f}\in \tors{A}$ and then write $\Tpair{T}{F}$ for the associated torsion pair in $\Cosilt{A}$, and use the  notation $(\Z_{\mathbf{t}}, \I_\mathbf{t})$  and $\N_{\mathbf{t}}$ for the other mathematical objects. 
\end{notation}

\subsection{Relationship with cosilting modules and cosilting complexes}\label{sec: cosilting}
Cosilting complexes in $\Der{A}$ were introduced in  \cite{ZhangWei:17, NicolasSaorinZvonareva:19, PsaroudakisVitoria:18}, dualising and generalising the silting complexes in $\mathrm{K}^{\mathrm{b}}(\mathrm{proj}A)$ introduced by Keller and Vossieck \cite{KellerVossieck:88}.

\begin{definition} A complex $\sigma \in \Kb{A}$ is a \textbf{cosilting complex} if    $\Hom{\Der{A}}(\sigma^I, \sigma[1]) = 0$ for all sets $I$, and $\Kb{A}$ is the smallest thick subcategory of $\Der{A}$ containing $\Prod{\sigma}$. \end{definition}

The zeroth cohomologies of two-term cosilting complexes were characterised as \textbf{cosilting modules} \cite{BreazPop:17} which generalise cotilting modules \cite{CDT}. We take this as the definition of cosilting module.

The maximal rigid subsets of $\K{A}$ and the cosilting pairs in $\Mod{A}$ for a left artinian ring $A$ should be thought of as the sets given by the indecomposable summands of two-term cosilting complexes in $\Der{A}$ and cosilting modules in $\Mod{A}$ respectively. This is not precisely true because we are always considering cosilting complexes and cosilting modules up to equivalence: We say that $\sigma$ and $\mu$ are {\bf equivalent}, written $\sigma\sim \mu$, if $\Prod{\sigma} = \Prod{\mu}$.

 Here is the precise statement establishing the link  to cosilting complexes and cosilting modules.

\begin{theorem}\label{bij_cosilt mod cpx}{\cite[Thm.~3.8, Thm.~2.17 and 4.13, Cor.~4.9]{ALS1}}
Let $A$ be a left artinian ring.  
 
There is a bijection  \[ \CosiltZg{A} \to \left\{\sigma \in \K{A} \mid \sigma \text{ cosilting complex in } \Der{A}\right\}/_{\sim}\] given by the assignment $\N\mapsto \sigma_\N:=\prod_{\mu \in \N} \mu$ with inverse assignment $\sigma\mapsto \Ind{\Prod{\sigma}}$.

There is a bijection given by \[\psi \colon \Cosiltpair{A} \to \left\{ C \in \Mod{A} \mid C \text{ cosilting module in } \Mod{A}\right\}/_{\sim}\] where $\psi((\Z, \I)) := \prod_{X\in \Z} X$ and $\psi^{-1}(C)= (\Z_C, \I_C)$ is given by $\Z_C:=\Ind{\Prod{C}}$ and $\I_C := \Z_C^{\perp_0}\cap\InjInd{A}$. If $A$ is an Artin algebra, then $\I_C = C^{\perp_0}\cap\InjInd{A}$.

In particular, $\Prod{\N}=\Prod{\sigma_\N}$, and the  torsion pair of finite type $\Tpair{T}{F}$ associated to the maximal rigid set $\N$  has torsion-free class  $\mathcal F= \Cogen{\H{}(\N)}=\Ctf{\N}$.  
\end{theorem}

The upshot of this theorem is that we can easily apply results that hold for cosilting complexes and cosilting modules to maximal rigid sets and cosilting pairs. Indeed, we will often cite results known for cosilting complexes or modules, leaving it to the reader to apply the above theorem. To illustrate how this works, we give an explicit proof of the following result.

\begin{definition}
A module $C\in\Mod{A}$ is called \textbf{cotilting} (or sometimes \textbf{1-cotilting}) if $\Cogen{C} = {}^{\perp_1}C$.  For a cotilting module $C$, the pair $({}^{\perp_0}C, \Cogen{C} = {}^{\perp_1}C)$ is a torsion pair of finite type in $\Mod{A}$.  
The torsion-free class $\Cogen{C}$ is called a \textbf{cotilting class}.
\end{definition}

Let $\Cpair{Z}{I} \in \Cosiltpair{A}$.  We fix the notation $\bar{A} := A/\Ann{\Z}$ and recall that $\Mod{\bar{A}}$ can be identified with the full subcategory of $\Mod{A}$ consisting of modules $M$ such that $\Ann{\Z}\cdot M = 0$.  As such, we have that $\F= \Cogen{\Z}\subseteq \Mod{\bar{A}}$.

\begin{proposition}\label{prop: cotilt/cosilt}
Let $\Cpair{Z}{I}\in\Cosiltpair{A}$ and let $ \Tpair{T}{F}$ be the  torsion pair of finite type assigned to $\Cpair{Z}{I}$ in Theorem \ref{bij_posets}(2).  Then the module $C = \prod_{X\in\Z}X$ is a cotilting $\bar{A}$-module with associated torsion pair $(\bar{\T}, \bar{\F}) := (\T\cap\Mod{\bar{A}}, \F)$ in $\Mod{\bar{A}}$.
\end{proposition}
\begin{proof}
First we observe that $\Ann{\Z} = \Ann{C}$ and that the module $C$ is a cosilting $A$-module by Theorem \ref{bij_cosilt mod cpx}. By \cite[Thm.~3.6]{Angeleri:18}, the module $C$ is a cotilting $\bar{A}$-module.  The associated torsion-free class in $\Mod{\bar{A}}$ consists of all $\bar{A}$-modules cogenerated by $\Prod{C}$. Since we have equalities $\Cogen{C} = \Cogen{\Z} = \F$ inside $\Mod{A}$ and $\F \subseteq \Mod{\bar{A}}$, it is clear that the  torsion pair is of the form stated.
\end{proof}

We will need the following lemma in subsequent sections.  If $(\Z, \I)$ is a cosilting pair with associated torsion pair of finite type $(\T,\F)$, then we know from \cite[Thm.~3.5]{BreazZemlicka:18} that every  $A$-module admits an $\F$-cover. It turns out that we can detect whether an $A$-module  belongs to $\Mod{\bar{A}}$ by properties of its  $\F$-cover, where $\bar{A}:= A/\Ann{\Z}$.

\begin{lemma}\label{lem: surj cover}
An $A$-module $M$ has a surjective $\mathcal{F}$-cover if and only if $M$ is contained in $\Mod{\bar{A}}$.
\end{lemma} 
\begin{proof}
First suppose that $M$ has a surjective $\mathcal{F}$-cover.  Then there exists a short exact sequence \[0\to K_M \to F_M \overset{f}{\to} M \to 0\] such that $f$ is an $\mathcal{F}$-cover of $M$.  It follows that $M$ is contained in $\Mod{\bar{A}}$, as $ F_M \in \Mod{\bar{A}} $ and this category is closed under quotients.
Now assume that $M$ is contained in $\Mod{\bar{A}}$.  Then, since $\mathcal{F}$ is a cotilting class, it is generating in $\Mod{\bar{A}}$ and hence every module in $\Mod{\bar{A}}$ has a surjective $\mathcal{F}$-cover. 
\end{proof}

\subsection{Cosilting and t-structures}
It is well-known that two-term cosilting complexes are deeply connected to HRS-tilted t-structures. According to Theorem \ref{bij_cosilt mod cpx}, maximal rigid sets and cosilting pairs enjoy analogous relationships with t-structures.  In this section we summarise the results that will be important to us.  We refer to  \cite{Laking:20} \cite{MarksVitoria:18}, \cite{MarksZvonareva:23} and \cite[Prop.~3.2]{ALS1} for details.

\begin{notation} If we denote 
\[\D^{\leq0} = \{X \in \Der{A} \mid \mathrm{H}^i(X) = 0 \text{ for all } i>0\} \text{ and }\D^{\geq0} = \{X \in \Der{A} \mid \mathrm{H}^i(X) = 0 \text{ for all } i<0\}\]
and set $\D^{<0}=\D^{\leq0}[1]$, we obtain a t-structure
$\ststr{} := (\D^{<0}, \D^{\geq 0})$ in $\Der{A}$,  called  the \textbf{standard t-structure}. 
		Note that $\H{\mathbb{D}}=\H{}$ is the standard 0th cohomology functor, and the heart $\heart{\mathbb{D}}$ is equivalent to $\Mod{A}$. 
		We  write $\ststr{}^b := (\D^{<0}\cap \Db{A},\: \D^{\geq 0}\cap\Db{A})$ for  the standard t-structure in $\Db{A}$; its heart is equivalent to $\mod{A}$.
	\end{notation}

The following construction due to Happel, Reiten and Smal{\o} allows to construct  new t-structures from torsion pairs in the heart.

\begin{definition}  \label{HRS}\cite{HappelReitenSmalo:96}
Let $\tstr{} = (\X, \Y)$ be a  t-structure in $\Der{A}$, and assume that $\tstr{}$ is non-degenerate, i.e.~$\bigcap_{i\in\mathbb{Z}}\X[i] = \{0\} = \bigcap_{i\in\mathbb{Z}}\Y[i]$.
Let $t = (\T, \F)$ be a torsion pair in the heart $\heart{\tstr{}}$.  The \textbf{(right) HRS-tilt} of $\tstr{}$ at $t$ is the t-structure $\tstr{t^-} = (\X_{t^-}, \Y_{t^-})$ given by 
	\[\X_{t^-} = \{ X \in \Der{A} \mid \H{{\tstr{}}}(X) \in \T \text{ and } \mathrm{H}^k_{\tstr{}}(X) = 0 \text{ for all } k>0\} \: \text{ and}\]
	\[\Y_{t^-} = \{ X \in \Der{A} \mid \H{{\tstr{}}}(X) \in \F \text{ and } \mathrm{H}^k_{\tstr{}}(X) = 0 \text{ for all } k<0\}.\]
The heart $\heart{t^-}$ of $\tstr{t^-}$ is the full subcategory 
	\[\heart{t^-} = \{X \in \Der{A} \mid \text{ there exists a triangle } F \to X \to T[-1] \to F[1] \text{ where } F\in\F \text{ and } T\in \T\}\]
and it contains a torsion pair $\bar{t} = (\F, \T[-1])$. Let $\H{t^-}$ denote the associated cohomological functor.
\end{definition}

We will consider HRS-tilts of the standard t-structure at the elements of $\Cosilt{A}$ and we will refer to them as \textbf{cosilting t-structures}. We summarise their properties and fix some notation. 

\begin{notation} As in Notation~\ref{fromt}, we start with a torsion pair  
$\tpair{t}{f}$  in $\tors{A}$ with  associated torsion pair  $\Tpair{T}{F}$ in $\Cosilt{A}$ and maximal rigid set $\N_\mathbf{t}\in \CosiltZg{A}$. This data gives rise to a t-structure with some useful properties. We fix some notation for this t-structure and state some such properties; we refer to the dual of \cite[Rem.~4.10]{AngeleriMarksVitoria:16} and \cite[Thm.~5.2]{Saorin:2017} for more details.
\begin{itemize}
\item  $(\X_\mathbf{t}, \Y_\mathbf{t}):=({}^{\perp_{\leq 0}}\N_\mathbf{t}, {}^{\perp_{> 0}}\N_\mathbf{t})$, which coincides with the  HRS-tilt of $\ststr{}$ at  $\Tpair{T}{F}$; 
\item $\heart{\mathbf{t}}$ for the heart of $(\X_\mathbf{t}, \Y_\mathbf{t})$, which is a locally coherent Grothendieck category;
\item $\H{\mathbf{t}}$ for the cohomological functor given by the t-structure $(\X_\mathbf{t}, \Y_\mathbf{t})$;
\item $\fp{\heart{\mathbf{t}}}$ for the subcategory  of finitely presented objects in $\heart{\mathbf{t}}$.
\end{itemize}
\end{notation}

Notice that  the HRS-tilt $(\X'_\mathbf{t}, \Y'_\mathbf{t})$ of the standard t-structure $\ststr{}^b$  at         $\tpair{t}{f}$ coincides with the  restriction  $(\X_\mathbf{t}\cap\ststr{}^b, \Y_\mathbf{t}\cap\ststr{}^b)$ of 
$(\X_\mathbf{t}, \Y_\mathbf{t})$ to $\ststr{}^b$, and its heart $\heart{\mathbf{t}}\cap\ststr{}^b$ equals
$\fp{\heart{\mathbf{t}}}$.
Moreover, when $A$ is a  finite-dimensional algebra over a field $\mathbb K$, it is well known that $\tpair{t}{f}$ is functorially finite if and only if
 $(\X'_\mathbf{t}, \Y'_\mathbf{t})$ is of the form  $(T^{\perp_{>0}}\cap\ststr{}^b, T^{\perp_{\leq 0}}\cap\ststr{}^b)$ for some silting complex $T\in\mathrm{K}^{\mathrm{b}}(\mathrm{proj}{A})$.
In this case
 $\N_\mathbf{t}$ is the set of indecomposable direct summands of $\nu T$ where $\nu $ is the Nakayama functor.

\subsection{Cosilting and the Ziegler spectrum}

In the paper \cite{ALS1}, strong links were made between the posets $\CosiltZg{A}$ and $\Cosiltpair{A}$ and the Ziegler spectra $\Zg{A}$ and $\Zg{\Der{A}}$ respectively. We summarise the important results here.

\begin{theorem}\label{Thm: cosilting and Zg}\cite[Prop.~3.2 and Cor.~3.3]{ALS1}
The following statements hold for a torsion pair  $\tpair{t}{f}$ in $\tors{A}$.
\begin{enumerate}
\item The set $\N_\mathbf{t}\in\CosiltZg{A}$ is a closed set in $\Zg{\Der{A}}$.
\item The set $\Z_\mathbf{t}$ is a closed set of $\Zg{A}$.
\item The cohomological functor $\H{\mathbf{t}}$ restricts to an equivalence of additive categories \[ \H{\mathbf{t}} \colon \Prod{\N_\mathbf{t}} \to \Inj{\heart{\mathbf{t}}}\]
and induces a homeomorphism \[\H{\mathbf{t}}\colon \N_\mathbf{t} \to \Sp{\heart{\mathbf{t}}}\] where $\N_\mathbf{t}$ has the subspace topology $\N_\mathbf{t} \subseteq \Zg{\Der{A}}$.
\end{enumerate}
\end{theorem}

\section{Mutation of maximal rigid sets}\label{Sec: mutation}

The mutation operation on cosilting complexes was introduced in \cite{ALSV} and irreducible mutations can be thought of as those corresponding to removing and replacing exactly one indecomposable summand of the cosilting complex. In our setting this can be reinterpreted as the following operation on maximal rigid  sets.

\begin{definition}\label{Def: mutation}
Let $\L$ and $\R$ be maximal rigid sets in $\Der{A}$.  We say that $\L$ and $\R$ are \textbf{related by mutation} if there is a bijection $\Delta \colon \L \to \R$ such that $\Delta(\mu) = \mu$ for every $\mu\in \L\cap\R$ and, for every $\lambda\in \L\setminus \R$, there exists a triangle 
	\begin{equation}\label{eq: exchange triangle}\Delta(\lambda) \overset{\Psi_\lambda}{\longrightarrow} \epsilon_\lambda \overset{\Phi_\lambda}{\longrightarrow} \lambda \to \Delta(\lambda)[1]\end{equation}
where $\Phi_\lambda \colon \epsilon_\lambda \to \lambda$ is a minimal right $\Prod{\L\cap\R}$-approximation of $\lambda$ and $\Psi_\lambda \colon \Delta(\lambda) \to \epsilon_\lambda$ is a minimal left $\Prod{\L\cap\R}$-approximation of $\Delta(\lambda)$.

We say that $\L$ is a \textbf{left mutation of $\R$ (at $\L\setminus\R$)} and that $\R$ is a \textbf{right mutation of $\L$ (at $\R\setminus\L$)}. We say that the mutation is \textbf{irreducible} if $|\L\setminus\R|= 1$.

Let $\mu \in \M$ be an element of a maximal rigid set $\M$. We say that $\mu$ is \textbf{mutable in $\M$} if there exists a maximal rigid set $\N = (\M\setminus\{\mu\})\cup\{\nu\}$ such that $\M$ and $\N$ are related by mutation.
\end{definition}

The setup in \cite{ALSV} is much more general than in Definition \ref{Def: mutation}; a cosilting complex $\sigma'$ is a right mutation of $\sigma$ if there exists a triangle $\epsilon_1 \to \epsilon_0 \overset{\Phi}{\to} \sigma \to \epsilon_1[1]$ such that $\Phi$ is a right $\mathscr{E}$-approximation where $\mathscr{E} = \Prod{\sigma}\cap\Prod{\sigma'}$ and $\epsilon_0 \oplus \epsilon_1$ is a cosilting object equivalent to $\sigma'$. The mutation is called irreducible if there is exactly one indecomposable object (up to isomorphism) in $\Prod{\sigma}$ that is not contained in $\mathscr{E}$. Irreducible left mutation of cosilting objects is defined dually and the two operations induce mutually inverse operations on equivalence classes of cosilting objects.   

We have already seen in Theorem \ref{bij_cosilt mod cpx} that a maximal rigid set $\N$ determines a two-term cosilting object $\sigma_\N = \prod_{\mu\in\N}\mu$ in $\Der{A}$ and that $\N$ is determined uniquely from the equivalence class of $\sigma_\N$. Mutation of maximal rigid sets is nothing but mutation of the associated cosilting complexes.

\begin{lemma} Given  $\R$ and $\L$ in $\CosiltZg{A}$, we have that 
\begin{enumerate}
\item $\Prod{\L\cap\R} = \Prod\L \cap \Prod\R$;
\item $\sigma_{\R}$ is a right mutation of $\sigma_\L$ if and only if $\R$ is a right mutation of $\L$;  
 
\end{enumerate}
\end{lemma}
\begin{proof}
(1)  We will use the fact that $\L$ and $\R$ are closed sets of $\Zg{\Der{A}}$ and show the corresponding result in the functor category $\Mod{\Der{A}^c}$. Recall that $\Mod{\Der{A}^c}$ is the category of additive contravariant functors from the category $\Der{A}^c$  of compact objects in $\Der{A}$ to the category of abelian groups. It is well known that $\Mod{\Der{A}^c}$ is a locally coherent Grothendieck category. The 
	isomorphism classes of indecomposable injective objects in  $\Mod{\Der{A}^c}$ form a topological space denoted by $\Sp{\Mod{\Der{A}^\c}}$ with a basis of open sets  given
by the collection of sets
	\[\open{C} = \{E \in \Sp{\Mod{\Der{A}^c}} \mid \Hom{\Mod{\Der{A}^c}}(C, E) \neq 0 \}\] 
where $C$ runs through the finitely presented objects of  $\Mod{\Der{A}^c}$. By {\cite[Fundamental Correspondence]{Krause:02}} there is a homeomorphism 
	\[ \y \colon \Zg{\Der{A}} \to \Sp{\Mod{\Der{A}^c}},\]
	so we can work in $\Sp{\Mod{\Der{A}^c}}$ rather than $\Zg{\Der{A}}$. We note that $\y$ commutes with both direct summands and direct products.
Moreover, as shown in
{\cite{Herzog:97, Krause:97}}, there is a bijection between the closed subsets  of $\Sp{\Mod{\Der{A}^c}}$ and the hereditary torsion pairs $(\T, \F)$ of finite type in $\Mod{\Der{A}^c}$ given by the assignments
	\[\Z \mapsto ({}^{\perp_0}\Z, \Cogen{\Z}) \hspace{10mm} (\T, \F) \mapsto \F \cap \Sp{\Mod{\Der{A}^c}}. \] 
	
 Consider now two closed subsets $\Z_1,\Z_2$ in $\Sp{\Mod{\Der{A}^c}}$. Each  $\Z_i$ is associated with a hereditary torsion pair of finite type in $\Mod{\Der{A}^c}$ with torsion-free class $\Cogen{\Z_i$}. 
Now $\Cogen{\Z_1}\cap\Cogen{\Z_2}$ is also a torsion-free class of finite type, and it is associated with the closed subset $\Cogen{\Z_1}\cap\Cogen{\Z_2}\cap\Sp{\Mod{\Der{A}^c}}=\Z_1\cap\Z_2$.  In other words we have that $\Cogen{\Z_1}\cap\Cogen{\Z_2} = \Cogen{\Z_1\cap\Z_2}$. Noting that $\Prod{\Z}$ is exactly the class of injective objects in $\Cogen{\Z}$ for any $\Z \subseteq \Sp{\Mod{\Der{A}^c}}$, we conclude that $\Prod{\Z_1}\cap\Prod{\Z_2}=\Prod{\Z_1\cap\Z_2}$.

(2)  
The `only-if' part of the statement follows from \cite[Prop.~4.12]{ALSV} and its proof, and the `if' part holds because a right $\Prod{\L\cap \R}$-approximation of $\sigma_\L$ is given by the product of the triangles in (\ref{eq: exchange triangle}) with the triangle 
	\[0 \to {\prod_{\nu \in \L\cap\R} \nu} \overset{\mathrm{id}}{\to} \prod_{\nu \in \L\cap\R}\nu \to 0[1].\]  

The direct sum of the two left-most term of this triangle is given by $\sigma := \sigma_\R \oplus \prod_{\lambda \in \L\setminus \R}\epsilon_\lambda$ and $\Prod{\sigma}=\Prod{\sigma_\R}$.  It follows from \cite[Lem.~3.6]{ALS1} and \cite[Prop.~3.10]{MarksVitoria:18} that $\sigma$ is a cosilting complex equivalent to $\sigma_\R$.
	
Noting that $\Prod{\L\cap\R} = \Prod\L \cap \Prod\R$ by (1), we conclude that $\sigma$ is a right mutation of $\sigma_\L$ and is equivalent to $\sigma_\R$.

\end{proof}

Mutation of maximal rigid sets can be interpreted from different angles: as an HRS-tilt of the associated cosilting t-structure, as an operation that enlarges the associated torsion pair by a wide subcategory, or as an operation exchanging injective envelopes of simples in the associated hearts. Let us explain these points of view.

\subsection{Mutation and HRS-tilting}

In \cite[Sec.~3-4]{ALSV} it was shown that two cosilting complexes are related by mutation if and only if the t-structures determined by them are related by an HRS-tilt of a particular form.  Using the principles of Section \ref{sec: cosilting}, we can make the same statement about two maximal rigid sets. We summarise the main results in the following theorem.

\begin{theorem}
\label{Thm: serre hereditary mutat} \cite[Thm.~3.5 and 7.8, Prop.~7.5]{ALSV}  
 Let $\tpair{u}{v} \leq\tpair{t}{f}$ be torsion pairs in $\tors{A}$.  
\begin{enumerate}

\item The t-structure $(\X_\mathbf{t}, \Y_\mathbf{t})$ in $\Der{A}$ is a right HRS-tilt of  the t-structure $(\X_\mathbf{u}, \Y_\mathbf{u})$ at a torsion pair $(\S, \R)$ in the heart $\heart{\mathbf{u}}$  with $\S = \T\cap \V$. 
\item The following statements are equivalent:
\begin{enumerate}
\item The maximal rigid sets $\N_\mathbf{t}$ and $\N_\mathbf{u}$ are related by mutation.

\item The torsion pair $(\S, \R)$ in $\heart{\mathbf{u}}$ is hereditary, i.e., $\S$ is a localising subcategory of $\heart{\mathbf{u}}$.
\item The torsion pair $(\S, \R)$ in the heart $\heart{\mathbf{u}}$  coincides with the hereditary torsion pair $({}^{\perp_0}\H{\mathbf{u}}(\E), \Cogen{\H{\mathbf{u}}(\E)})$ where $\E = {\N_\mathbf{t}}\cap {\N_\mathbf{u}}$.
\end{enumerate}
\end{enumerate}
\end{theorem}

\subsection{Mutation and wide intervals}\label{sec: wide intervals}

In \cite[Sec.~7-9]{ALSV}, a strong relationship is established between mutation of cosilting complexes and so-called wide intervals of $\tors{A}$ and $\Cosilt{A}$. In \cite{AsaiPfeifer}, it is shown that the larger torsion class in the interval can be obtained from the smaller one by `enlarging' it by a wide subcategory. Using Theorem \ref{bij_cosilt mod cpx}, we can translate these results into theorems about maximal rigid sets and cosilting pairs.  In this subsection, we summarise the particular such results that we will use later in the article.

\begin{definition}
A \textbf{wide} subcategory of an abelian subcategory is one that is closed under kernels, cokernels and extensions. 
Two torsion pairs $\tpair{u}{v} \leq\tpair{t}{f}$ in $\tors{A}$ form a \textbf{wide  interval} in $\tors{A}$ if the intersection $\mathbf{t}\cap\mathbf{v}$ is a wide subcategory of $\mod{A}$. We will abuse notation slightly and denote these intervals by $[\mathbf{u}, \mathbf{t}]$.

\end{definition}

\begin{theorem}\cite[Thm.~8.8]{ALSV}\label{Thm: wide}
Two torsion pairs $\tpair{u}{v} \leq\tpair{t}{f}$ in $\tors{A}$ form a wide interval $[\mathbf{u}, \mathbf{t}]$ if and only if the maximal rigid sets $\N_\mathbf{t}$ and $\N_\mathbf{u}$ are related by mutation.
\end{theorem}

It was shown by Ringel that wide subcategories of a length category are again length categories and so their simple objects play an important role.  They form a collection of Hom-orthogonal bricks, called a \textbf{semibrick}. Given a wide interval $[\mathbf{u}, \mathbf{t}]$  in $\tors{A}$,   the semibrick given by the simple objects in $\mathbf{t}\cap \mathbf{v}$  controls the mutation process, as we are going to see next.

\subsection{Mutation and simple objects}\label{sec: mutation simples}

The next result shows that mutation of maximal rigid subsets can be regarded as exchanging injective envelopes of finitely presented simple objects in hearts via the homeomorphism in Theorem~\ref{Thm: cosilting and Zg}(3). To formulate this we need the following terminology.

\begin{definition} Let $\tpair{t}{f}\in \tors{A}$ and let $S \in \heart{\mathbf{t}}$ be a simple object in the heart $\heart{\mathbf{t}}$.  The {\textbf{$\N_\mathbf{t}$-injective} envelope of $S$} is defined to be the unique element $\mu \in \N_\mathbf{t}$ such that $\Hom{\Der{A}}(S, \mu) \neq 0$.
\end{definition}

\begin{remark}      \label{Rmk: N-inj env are inj env}
The $\mathcal{N_\mathbf{t}}$-injective envelope exists of any simple $S\in\heart{\mathbf{t}}$.  Indeed, the injective envelope $E$ of $S$ exists because $\heart{\mathbf{t}}$ is a Grothendieck category and it is indecomposable because $S$ is simple.  Thus the homeomorphism $H^0_{\mathbf{t}} \colon \N_{\mathbf{t}} \to \Sp{\heart{\mathbf{t}}}$  yields $\mu\in \N_{\mathbf{t}}$ such that $H^0_{\mathbf{t}}(\mu)\cong E$.  Moreover $\Hom{\Der{A}}(S, \mu')\cong \Hom{\heart{\mathbf{t}}}(S, H^0_{\mathbf{t}}(\mu'))$ for any $\mu'\in\N_{\mathbf{t}}$ and so $\mu$ is the unique element of $\N_{\mathbf{t}}$ such that $\Hom{\Der{A}}(S, \mu) \neq 0$ because $E$ enjoys the same property among the elements of $\Sp{\heart{\mathbf{t}}}$.
\end{remark}

\begin{theorem}[{\cite[Thm.~9.7, Lem.~9.11]{ALSV}}]
\label{Thm: simple mutation}

 Let $\tpair{u}{v} \leq\tpair{t}{f}$ be torsion pairs in $\tors{A}$.  
 The following statements are equivalent.	
 \begin{enumerate}
\item The maximal rigid sets $\N_\mathbf{t}$ and $\N_\mathbf{u}$ are related by mutation.
\item There exists a semibrick $\Omega_\mathbf{u}$  in $\mod{A}$ such that $\mathbf{t} \cap \mathbf{v} = \filt{\Omega_\mathbf{u}}$ in $\fp{\heart{}}_{\mathbf{u}}$ and every $S\in\Omega_\mathbf{u}$ is simple in $\fp{\heart{}}_{\mathbf{u}}$.
\item There exists a semibrick $\Omega_\mathbf{t}$ in $\mod{A}$ such that $\mathbf{t}[-1] \cap \mathbf{v}[-1] = \filt{\Omega_\mathbf{t}[-1]}$ in $\fp{\heart{}}_{\mathbf{t}}$ and every $S[-1]\in\Omega_\mathbf{t}[-1]$ is simple in $\fp{\heart{}}_{\mathbf{t}}$.
\end{enumerate}
In this case $\Omega := \Omega_\mathbf{u} = \Omega_\mathbf{t}$, and the bijection $\Delta\colon \N_\mathbf{u}\to \N_\mathbf{t}$  witnessing the mutation takes the $\N_\mathbf{u}$-injective envelope of $S$ to the $\N_\mathbf{t}$-injective envelope of $S[-1]$ for each $S\in\Omega$  and fixes the remaining elements of $\N_\mathbf{u}$.

\end{theorem}

We then have the following direct corollary for irreducible mutations.

\begin{definition}
Let $ (L, \le ) $ be a poset, $ x, y \in L $.  We say that $ y $ \textbf{covers} $ x $ if $ x < y $ and for any $ z \in L $ such that $ x \le z \le y $, either $ z = x $ or $ z = y $.
\end{definition}

We will sometimes refer to the \textbf{Hasse quiver} of the poset $L$. This is the directed graph where the vertices of the graph are the elements of $L$ and there is an arrow from $y$ to $x$ whenever $y$ covers $x$. In the case where the poset is $\tors{A}$, we will simplify notation and write $\mathbf{t} \to \mathbf{u}$ whenever $ \tpair{t}{f}$ covers $\tpair{u}{v}$.

\begin{corollary}\label{Thm: mutation}  Let $\tpair{u}{v} \leq\tpair{t}{f}$ be torsion pairs in $\tors{A}$.    The following statements are equivalent.
\begin{enumerate}
\item The maximal rigid sets $\N_\mathbf{t}$ and $\N_\mathbf{u}$ are related by irreducible mutation.
\item There is a module $S$ in $\mod{A}$ that is simple in $\fp{\heart{}}_{\mathbf{u}}$ such that $\mathbf{t} \cap \mathbf{v} = \filt{S}$.
\item There is a module $S$ in $\mod{A}$ that is simple in $\fp{\heart{}}_{\mathbf{t}}$ such that $\mathbf{t}[-1] \cap \mathbf{v}[-1] = \filt{S[-1]}$.
\item $\N_\mathbf{t}$ covers $\N_\mathbf{u}$ in $\CosiltZg{A}$.
\item $ \tpair{t}{f}$ covers $\tpair{u}{v}$  in $\tors{A}$.
\end{enumerate}

Moreover, under the equivalent conditions above, if $\N_\mathbf{u}\setminus\N_\mathbf{t} = \{\lambda\}$ 
and $\N_\mathbf{t}\setminus\N_\mathbf{u} = \{\rho\}$, then 
 $\lambda$ is the $\N_\mathbf{u}$-injective envelope of the simple object $S$ in $\heart{\mathbf{u}}$  and $\rho$ is the $\N_\mathbf{t}$-injective envelope of the simple object $S[-1]$ in $\heart{\mathbf{t}}$.

In particular, for any $\mathbf{t}\in\tors{A}$, an element $\mu\in\N_\mathbf{t}$ is mutable in $\N_\mathbf{t}$ if and only if it is the $\N_\mathbf{t}$-injective envelope of a finitely presented simple object in $\heart{\mathbf{t}}$.
\end{corollary}

 The equivalence with statement (4) can be found in \cite[Cor.~9.14]{ALSV}.
In order to prove it, it was necessary to connect simples in the HRS-tilted heart to the brick labels in the Hasse quiver of $\tors{A}$, or in other words, to the minimal extending and minimal coextending modules  considered in \cite{BarnardCarrollZhu:19}.  The objects in the following definition coincide with minimal extending and minimal coextending modules when $\A=\mod{A}$. 

 \begin{definition}\label{almost}
Let $\A$ be an abelian category and let $t = (\T, \F)$ be a torsion pair in $\A$. 
We say that $ B $ is \textbf{torsion almost torsion-free} with respect to $t$, if:
\begin{enumerate}
\item[(a)] $ B \in \mathcal{T} $, but every proper subobject of $ B $ is contained in $ \mathcal{ F } $; 
\item[(b)] for every short exact sequence $ 0 \to M \to T \to B \to 0 $, if $ T \in \mathcal{T} $, then $ M \in \mathcal{T} $.
\end{enumerate}
Dually, one defines   \textbf{torsion-free almost torsion} modules with respect to $t$.
\end{definition}

The next result connects these objects to simples in the heart and shows  that the set $\Omega$ in Theorem~\ref{Thm: simple mutation} is precisely the semibrick arising from the wide interval $[\mathbf u, \mathbf t]$.

\begin{proposition}\label{simples in the heart} 
Let $\tpair{t}{f}\in\tors{A}$ and let $\Tpair{T}{F}$ be the  associated torsion pair in $\Cosilt{A}$. 
\begin{enumerate}
\item 
\cite[Thm.~3.6]{AHL}  The simple objects in  $\heart{\mathbf{t}}$  are precisely the objects of the form $T[-1]$  or $F$ where $T\in\Mod{A}$ is torsion almost torsion-free, and $F\in\Mod{A}$ is torsion-free almost torsion with respect to $\Tpair{T}{F}$. 
  The simple objects  in $\fp{\heart{}}_{\mathbf{t}}$ are those given by finitely generated $T$ and $F$.
\item
\cite[Prop.~2.11]{Sentieri:22} 
The modules that are torsion almost torsion-free with respect to $\Tpair{T}{F}$ are always finitely generated and coincide with the torsion almost torsion-free modules with respect to $\tpair{t}{f}$. The modules  that are  torsion-free almost torsion with respect to $\Tpair{T}{F}$ and are finitely generated coincide with the torsion-free almost torsion modules with respect to $\tpair{t}{f}$.
\item \cite[Lem.~9.6]{ALSV} Suppose $\tpair{u}{v}\in\tors{A}$ such that $[\mathbf u, \mathbf t]$ is a wide interval of $\tors{A}$.  The semibrick formed by the simple objects of $\mathbf{t}\cap \mathbf{v}$ consists of the modules  in $\mathbf v$ that are torsion almost torsion-free with respect to $\tpair{t}{f}$,  or equivalently, the modules in $\mathbf t$ that are torsion-free almost torsion with respect to $\tpair{u}{v}$.
\end{enumerate}
\end{proposition}

\section{Critical and special objects}\label{sec: crit+spec}
Throughout Section \ref{sec: crit+spec} we fix $\tpair{t}{f} \in  \tors{A}$ and we denote by $\Tpair{T}{F}\in \Cosilt{A}$, by $\Cpair{Z}{I} := (\mathcal{Z}_{\mathbf{t}}, \mathcal{I}_{\mathbf{t}}) \in \Cosiltpair{A}$ and by $\mathcal{N} :=  \mathcal N_{\mathbf{t}} \in \CosiltZg{A}$ the mathematical structures corresponding to $\tpair{t}{f}$ under the cosilting bijections.

By Theorem~\ref{bij_posets}, the modules  $N\in\Z$ and  $I\in\I$ correspond to complexes $\mu_N$ and  $I[-1]$ in $\N$. We aim to determine when these complexes   are $\N$-injective envelopes of simples, employing the description of the simple objects in  $\heart{\mathbf{t}}$ given above.

\begin{lemma}\label{lem: N-injectives in ModA}
Let $F\in\F$ be torsion-free almost torsion  and let $T\in\T$ be torsion almost torsion-free with respect to $\Tpair{T}{F}$.  The following statements hold for $N\in \Z$ and $I\in\I$.  
\begin{enumerate}
\item The complex $\mu_N$ is the $\N$-injective envelope of $F$ if and only if $\Hom{A}(F,N) \neq 0$.
\item The complex $I[-1]$ cannot be the $\N$-injective envelope of $F$.
\item The complex $\mu_N$ is the $\N$-injective envelope of $T[-1]$ if and only if $T$ is not in $\Ctf{\mu_N}$.
\item The complex $I[-1]$ is the $\N$-injective envelope of $T[-1]$ if and only if $\Hom{A}(T, I)\neq 0$.
\end{enumerate}
\end{lemma}
\begin{proof}
By Remark \ref{Rmk: N-inj env are inj env}, an element $\mu$ of $\N$ is an $\N$-injective envelope of $F$ (respectively $T[-1]$) if and only if there is a non-zero homomorphism $F \to \mu$ (respectively $T[-1]\to \mu$).  The statements then follow from a straightforward analysis of the possible chain maps up to homotopy.
\end{proof}

\begin{definition} 
(1) A morphism $ f : N \to \overline{N} $ in $ \mathcal{F} $ is \textbf{left almost split in $ \mathcal{F} $}  if it is not a split monomorphism and for every morphism $ g : N \to N' $ in $\F$ which is not a split monomorphism there exists $ h : \overline{N} \to N' $ such that $ g = h \circ f $. 

(2) We say that $N\in\F$ is \textbf{neg-isolated} in $\F$ if there exists a left almost split map $N\to\overline{N}$ in $\F$.

(3) We say that $N\in\F$ is \textbf{critical} in $\F$ if there exists a left almost split map $N\to\overline{N}$ in $\F$ that is an epimorphism.

(4) A module $N\in\F\cap\F^{\perp_1}$ is called \textbf{special} in $\F$ if there exists a left almost split map  $N\to\overline{N}$ in $\F$ that is a monomorphism.
\end{definition}

\begin{remark}\label{Rem: critical or special}
Clearly, there exists a left almost split map $ N \to \overline{N}$ in $\mathcal{F}$ that is a monomorphism if and only if every left almost split map starting at $N$ is a monomorphism. If this is not the case, then there is a left almost split map $ N \to \overline{N}$ in $\mathcal{F}$ that is an epimorphism,  as shown in \cite[Lem.~4.3]{AHL}. In other words, every module $N\in \F\cap\F^{\perp_1}$ that is neg-isolated in $\F$ is either critical or special in $\F$. Morover, every critical module in $\F$ is contained in $\F \cap \F^{\perp_1}$ by \cite[Prop.~5.16]{AHL}.
\end{remark}

The interplay between critical and torsion-free almost torsion modules is worked out in \cite{ALS1}.  The following theorem summarises the main result.

\begin{theorem}[{\cite[Thm.~5.11]{ALS1}}]\label{Thm: critical tf/t}
There is a bijection between the torsion-free almost torsion modules with respect to $\Tpair{T}{F}$ and the critical modules in $\F$.  The bijection assigns to $F$  the unique module $N\in\Z$ such that $\Hom{A}(F, N)\neq 0$.
\end{theorem}

Following the terminology above, the $\N$-injective envelope of the simple object $F$ in $\heart{\mathbf{t}}$ is $\mu_N$ for the corresponding critical module $N\in\Z$.  

We now consider the $\N$-injective envelope of $T[-1]$ for $T$ torsion almost torsion-free with respect to $\Tpair{T}{F}$.  We will follow a similar line of reasoning to \cite[Sec.~5]{ALS1}. In particular, we will compare the torsion almost torsion-free modules and special modules with respect to $\Tpair{T}{F}$ in $\Mod{A}$ with those with respect to $(\bar{\T}, \bar{\F})$ in $\Mod{\bar{A}}$. See Proposition \ref{prop: cotilt/cosilt} for an explanation of the notation.

\begin{proposition}\label{prop: upgrade t/tf special}
\begin{enumerate}
\item An $A$-module is torsion almost torsion-free with respect to $(\bar{\T}, \bar{\F})$ if and only if it is torsion almost torsion-free with respect to $\Tpair{T}{F}$ and it admits a surjective $ \mathcal{F}-$cover.
\item  Consider a short exact sequence 
\[
\begin{tikzcd}
0 \arrow[r] & N \arrow[r, "b"] & \overline{N} \arrow[r, "f"] & T \arrow[r] & 0
\end{tikzcd}
\] in $\Mod{A}$.  We have that $ T $ is torsion almost torsion-free with respect to $ \Tpair{T}{F} $ and $ f $ is a $ \F-$cover if and only if $N$ is special in $\F$ and $b$ is a  left almost split map in $\F$.  In this case the module $N$ is indecomposable and hence $N\in\Z$.
 
\end{enumerate}
\end{proposition}
\begin{proof}
(1) Recall from  Lemma \ref{lem: surj cover} that an $A$-module $ B $ admits a surjective cover if and only if it is contained in $ \Mod{\bar{A}} $. 
It is clear that, in this situation, if $ B $ is torsion almost torsion-free with respect to $(\T, \F)$, it is also torsion almost torsion-free with respect to $(\bar{\T}, \bar{\F}) $. 

For the converse, let $B$ be torsion almost torsion-free with respect to $(\bar{\T}, \bar{\F})$ and note that $B\in\bar{\T}$ has a surjective $\F$-cover by Lemma \ref{lem: surj cover}. Notice that $ B $ is surely a torsion module and that all its proper submodules are contained in $ \bar{\F} = \mathcal{F} $. Thus it   only remains  to check condition (b) in Definition~\ref{almost}. Let $ 0 \to K \to T \to B \to 0 $ be an exact sequence in $ \Mod{A} $ with $ T \in \mathcal{T} $. We start with the case when $ K $ is torsion-free. Then using the surjective cover  $ f : F \to B $ of $ B $, we can construct the following commutative diagram:
\[
\begin{tikzcd}
0 \arrow[r] & K \arrow[d, equals] \arrow[r] & P \arrow[r] \arrow[d, two heads] & F \arrow[r] \arrow[d, "f", two heads] & 0 \\
0 \arrow[r] & K \arrow[r] & T \arrow[r] & B \arrow[r] & 0 
\end{tikzcd}
\]
By assumption $ K $ and $ F $ are in $ \mathcal{F} $, thus $ P \in \mathcal{F} $ and therefore $ T $ has a surjective cover. This forces $ T $ to be in $ \Mod{\bar{A}} $ by Lemma \ref{lem: surj cover}. Since $B$ is torsion almost torsion-free, we conclude that $ K \in \bar{\T} $, that is, $ K \in  \bar{\T} \cap \mathcal{F} = 0 $. 
In the general case, let $\operatorname{t}K$ denote the torsion part of $K$, and consider the pushout diagram
\[
\begin{tikzcd}
0 \arrow[r] & K \arrow[r] \arrow[d, two heads] & T \arrow[r] \arrow[d] & B \arrow[r] \arrow[d, equals] & 0 \\
0 \arrow[r] & K/\operatorname{t}K \arrow[r] & T' \arrow[r] & B \arrow[r] & 0 
\end{tikzcd}
\]
Apply the special case to the lower sequence to conclude that $ K = \operatorname{t}K $ is a torsion module.
 This shows that $ B $ is torsion almost torsion-free with respect to $(\T, \F)$.
 
 (2) The proof is completely analogous to the proof of \cite[Prop.~5.10(2)]{ALS1}.
\end{proof}

It remains to consider the torsion almost torsion-free modules whose $\F$-cover is not surjective.

\begin{proposition}
\label{prop:tTfHaveUniqueMax}  
Let $ T $ be a torsion almost torsion-free module with respect to $\Tpair{T}{F}$ whose $ \mathcal{F}$-cover is not surjective. 
Then $ T $ has a unique maximal submodule $L$, and there is a short exact sequence
\[
\begin{tikzcd}
0 \arrow[r] & L \arrow[r, "f"] & T \arrow[r] & S \arrow[r] & 0
\end{tikzcd}
\] 
where $f$ is an $\F$-cover in $\Mod{A}$ and the simple $A$-module $ S$ belongs to $ \mathcal{F}^{\perp_0} $.  Moreover,  $ T \in \mathcal{C}_{\mu_N} $ for all $N\in\Z$ and the injective envelope $I$ of $S$ is the unique  module  in $\mathcal{I}$ such that $ \Hom{A}(T, I) \ne 0 $. 
\end{proposition}
\begin{proof}
By Proposition \ref{simples in the heart}(2), $T$ is finitely generated. Thus $T$ has at least one maximal submodule. If $ T $ had two maximal submodules $ F_1 \ne F_2 $, then the summation map $ F_1 \oplus F_2 \to T $ would be surjective, contradicting the assumption we made on the $ \mathcal{F}-$cover. 

Hence $T$ has a unique maximal submodule $L$, and the embedding $f:L\hookrightarrow T$ is an $\F$-cover. This shows the existence of a simple module $ S $ fitting in the short exact sequence. If we had a non-zero morphism from a torsion-free module to $ S $, then it would be an epimorphism, and we could obtain a surjective map from a torsion-free module to $ T $ via pullback. This would again contradict the assumption, thus $ S \in \mathcal{F}^{\perp_0} $.

Next we prove the last statement so let $I$ be the injective envelope of $S$.  By Theorem \ref{bij_posets}, we have that $\I$ is the set of indecomposable injective $A$-modules that are contained in $\Z^{\perp_0}$.  Let $N\in\Z$ and consider $c\colon N\to I$.  Then $S$ must be contained in the image of $c$ and so the inverse image $F$ of $S$ under $c$ is a torsion-free module and the restriction $F \to S$ is non-zero.  This contradicts the previous paragraph so $I$ must be contained in $\I$.  Since $\Hom{A}(T, I)\neq 0$, it follows from Lemma \ref{lem: N-injectives in ModA}(4) that $I[-1]$ is the $\N$-injective envelope of $T[-1]$.  The uniqueness and the fact that $T\in \Ctf{\mu_N}$ for all $N\in\Z$ then follows from Lemma \ref{lem: N-injectives in ModA}(3) and the uniqueness of $\N$-injective envelopes.
\end{proof}

It follows from the previous proposition and Lemma \ref{lem: N-injectives in ModA}, that $I[-1]$ is the $\N$-injective envelope of $T[-1]$. The next proposition will allow us to identify which elements of $\I$ arise in this way via properties of their simple socles.

\begin{proposition}\label{prop: special injectives}
Consider a short exact sequence \[ 0 \longrightarrow L \overset{f}{\longrightarrow}  T \overset{g}{\longrightarrow} S \longrightarrow 0 \] in $\Mod{A}$.  We have that $ T $ is torsion almost torsion-free with respect to $ \Tpair{T}{F} $ and $ f $ is an $ \F-$cover if and only if $S$ is contained in $\F^{\perp_0}$, and the map $g$ has domain $T\in \T$ and kernel $L\in \F$, and  is universal with these properties, i.e.~for all non-zero $g' \colon T' {\rightarrow} S$ with $T'\in \T$ and $\ker{g'}\in \F$, there exists $h\colon T \to T'$ such that $g'h = g$.\end{proposition}
\begin{proof} We begin by noting that, for any $h \colon Y \to X$ with $Y,X \in \T$, the morphism $h[-1]$ is a monomorphism in $\heart{\mathbf{t}}$ if and only if $\ker{h} \in \F$ (see, for example, \cite[Lem.~2.3]{AHL}). Since $S$ is contained in $\T$ and the class $\T[-1]$ is closed under subobjects in $\heart{\mathbf{t}}$, the morphisms $g'$ in the statement correspond to the collection of non-zero subobjects $g'[-1] \colon T'[-1] \to S[-1]$ of $S[-1]$ in $\heart{\mathbf{t}}$.

Suppose that $T$ is torsion almost torsion-free with respect to $\Tpair{T}{F}$ and $f$ is an $\F$-cover. Clearly $T$ is contained in $\T$ and $L$ is contained in $\F$. Also $S$ is simple and contained in $\F^{\perp_0}$ by Proposition \ref{prop:tTfHaveUniqueMax}. Moreover $I[-1]$ is the $\N$-injective envelope of $T[-1]$ where $I$ is the injective envelope of $S$ in $\Mod{A}$. In particular, there is an essential monomorphism \[T[-1] \overset{g[-1]}{\longrightarrow} S[-1] \to \H{\mathbf{t}}(I[-1]).\]  It follows that $g[-1]$ is an essential monomorphism with simple domain and the required factorisation property for $g$ then follows easily.

Suppose that $S\in\F^{\perp_0}$, $T\in \T$, $L\in \F$ and $g$ has the property that, for all non-zero $g' \colon T' {\rightarrow} S$ with $T'\in \T$ and $\ker{g'}\in \F$, there exists $h\colon T \to T'$ such that $g'h = g$. The condition on the morphism $g$ means that the subobject $g[-1] \colon T[-1] \to S[-1]$ is contained in all non-zero subobjects of $S[-1]$ in $\heart{\mathbf{t}}$. As a consequence, the object $T[-1]$ cannot have any proper non-zero subobjects. In other words, $T[-1]$ is simple in $\heart{\mathbf{t}}$ and hence $T$ is torsion almost torsion-free by \ref{simples in the heart}. Finally we consider an $\F$-cover $f' \colon F \to T$ and the induced map $h \colon L \to F$ such that $f = f'h$. Since $S\in \F^{\perp_0}$, we also have that $gf'=0$ and hence there is some $h' \colon F \to L$ such that $f' = fh'$. Then $fh'h = f'h = f$ implies that $h'h$ is the identity because $f$ is a monomorphism. Moreover $f'hh' = fh' = f'$ implies that $hh'$ is an isomorphism because $f'$ is minimal. Thus $h$ is an isomorphism and so $f$ is an $\F$-cover.
\end{proof}

\begin{definition}
We say that $I\in\I$ is \textbf{special} if its socle $S:= \mathrm{Soc}(I)$ belongs to $\F^{\perp_0}$ and there exists $g \colon T \to S$ which is universal with the property  that $T\in \T$ and $\ker{g} \in \F$.
\end{definition}

\begin{theorem}\label{all in Ical appear}
\begin{enumerate}
\item There is a bijection between torsion almost torsion-free modules $T$ with surjective $\F$-cover and special modules with respect to $\Tpair{T}{F}$.  The bijection assigns $T$ to the unique module $N\in\Z$ such that $T\notin \Ctf{\mu_N}$.  Moreover $\Hom{A}(T, I)=0$ for all $I\in\I$.

\item There is a bijection between torsion almost torsion-free modules $T$ without surjective $\F$-cover and special modules in $\I$. The bijection assigns to $T$ the unique module $I$ in $\I$ such that $\Hom{A}(T, I)\neq 0$. Moreover $T\in\Ctf{\mu_N}$ for all $N\in \Z$.
\end{enumerate}\end{theorem}
\begin{proof}
(1) In order to establish the first bijection,  we check that, given a short exact sequence $0\to N \to \overline{N} \to T \to 0$ as in Proposition~\ref{prop: upgrade t/tf special}(2) with $N$ special and $T$ torsion almost torsion-free with respect to $\Tpair{T}{F}$, we can characterise  $N$ as the unique element of $\Z$ such that $T\notin \C_{\mu_N}$.  By Lemma \ref{Csigma}(2), the module  $T$ lies in $\C_{\mu_X}$ for some module $X$ if and only if all submodules of $T$ lie in ${}^{\perp_1}X$.  So, we certainly have that $T \notin \C_{\mu_N}$ as  $\Ext{A}^1(T, N) \neq 0$.  The rest follows from Lemma \ref{lem: N-injectives in ModA} and the uniqueness of $\N$-injective envelopes (see Remark \ref{Rmk: N-inj env are inj env}).

(2) The existence of the second bijection follows from Propositions \ref{prop:tTfHaveUniqueMax} and \ref{prop: special injectives}. The final statement is a consequence of Lemma \ref{lem: N-injectives in ModA} and the uniqueness of $\N$-injective envelopes (see Remark \ref{Rmk: N-inj env are inj env}).
\end{proof}

In the case that $A$ is a finite-dimensional algebra, we can make use of \cite [Theorem 3.1]{DemonetIyamaJasso:19} in order to improve the previous theorem.

\begin{theorem}\label{Thm: I is special fd algebra}
If $A$ is a finite-dimensional algebra, then every element of $\I$ is special. In other words, there is a bijection between torsion almost torsion-free modules $T$ without surjective $\F$-cover and modules in $\I$. Moreover, $\I$ consists precisely of the indecomposable injective modules in $\F^{\perp_0}$.
\end{theorem}
\begin{proof}
Since we are dealing with a question about torsion almost torsion-free modules, we can work with the restricted torsion pair $\tpair{t}{f} $ in $ \mod{A} $, see  Proposition~\ref{simples in the heart}(2). We want to apply the (dual of the) techniques in \cite{Enomoto:21}. Let us briefly recall the relevant notions and results from there.
A module $B\in\mathbf{t}$ is $\mathbf{t}$-{\em simple} if it satisfies condition (a) in Definition~\ref{almost}, that is,  every proper subobject of $ B $ is contained in $ \mathbf{f}$. The $ \mathbf{t}-$simple modules form a so-called {\em epibrick}, i.e.~a collection of bricks with the only non-zero maps between them being epimorphisms. There is an obvious partial order on this epibrick by setting $T\le T'$ if there is an epimorphism $T'\to T$. Now the dual of \cite[Lem.~4.4]{Enomoto:21} combined with \cite[Prop.~3.10]{AngeleriSentieri:23+} states that the torsion almost torsion-free modules with respect to $\tpair{t}{f} $ are precisely the maximal objects in this epibrick. 

Let us now turn to our statement. Let $I\in\mathcal{I} $. First, notice that $\mathbf{v}= {}^{\perp_0}I \cap \mod{A} $ is a Serre subcategory of $ \mod{A} $. The torsion pair $ (\mathbf{u}, \mathbf{v}) $  is then functorially finite and  distinct from the trivial torsion pair $(0, \mod{A})$. By \cite [Theorem 3.1]{DemonetIyamaJasso:19} and  \cite[Thm.~2.11]{BarnardCarrollZhu:19},  there exists a torsion pair covered by $ ( \mathbf{u}, \mathbf{v})$ in $\mathbf{tors}(A)$, and  thus  a torsion almost torsion-free module $ U $ with respect to  $ ( \mathbf{u}, \mathbf{v})$. The $ \mathbf{v}-$cover of $U$ can't be surjective as  $ \mathbf{v} $ is closed under quotients and $ 0 \ne U \in \mathbf{u} $. If $(\U,\V)=( \varinjlim\mathbf{u}, \varinjlim\mathbf{v})$ is the corresponding  torsion pair in $\Cosilt{A}$, then the  $\V$-cover of $U$ is not surjective either. From Proposition \ref{prop:tTfHaveUniqueMax} we infer that $ U $ has  a unique maximal submodule, and there is a sequence
\[
0 \to V \to U \to S \to 0
\]
where $V\in\V$ and $S$ is a simple module. Observe that $S=\textrm{Soc}\,I$. Indeed, $ \Hom{A}(U, I) \ne 0 $, so $\textrm{Soc}\,I$ is a composition factor of $ U $, but it is  not a composition factor of any of its proper submodules.

Now $ U $ is a module in $ \mathbf{u} \subseteq  \mathbf{t} $, however it might not be $ \mathbf{t}-$simple. So, we consider the set $ \mathcal{S} $ consisting of $ \mathbf{t}-$simple modules occurring as quotients of $ U $ and having a non-zero morphism to $ S $. This set is clearly non-empty, as the simple module $ S $ itself satisfies these conditions. 
Let $ T $ be a module in $ \mathcal{S} $ of maximal length. We claim that $ T $ is torsion almost torsion-free with respect to $ \tpair{t}{f} $. It is $ \mathbf{t}-$simple by construction, thus it remains to show that it is maximal in the epibrick of $ \mathbf{t}-$simple modules.
To this end, we pick  a $ \mathbf{t}-$simple module $ T' $ with an epimorphism $ T' \to T $.
Consider the following pull-back diagram:
\[
\begin{tikzcd}
0 \arrow[r] & K \arrow[r] \arrow[d, equals] & P \arrow[d, two heads] \arrow[r] & U \arrow[r] \arrow[d, two heads] & 0 \\
0 \arrow[r] & K \arrow[r] & T' \arrow[r] & T \arrow[r] & 0
\end{tikzcd}
\]
By $\mathbf{t}-$simplicity of $ T' $, the kernel $ K \in \mathbf{f} \subseteq \mathbf{v} $. Since $ U $ is torsion almost torsion-free with respect to $ \tpair{u}{v}$, we have two possibilities: either $ P \in \mathbf{v} $ or the upper sequence splits (cf.~condition (3') in \cite[Def.~2.1]{Sentieri:22}). In the first case we reach a contradiction, since we would obtain a non-zero map $ P \to T' \to T \to S \to I$. In the second case we obtain a  non-zero map $ U \to T' $, which must be an epimorphism, as every proper submodule of $ T' $ is an element of $ \mathbf{f} \subseteq \mathbf{v} $. This means that $ T' $ is an element of the set $ \mathcal{S} $, thus, by maximality of the length of $ T $, we must have $ T' \simeq T $. 

We have shown that $ T $ is torsion almost torsion-free with respect to $\Tpair{T}{F}$, and by construction there is  a sequence
\[
\begin{tikzcd}
0 \arrow[r] & F \arrow[r, "f"] & T \arrow[r] & S \arrow[r] & 0
\end{tikzcd}
\] 
with $F\in\mathcal{F}$, where $f$ is an $\F$-cover since $\mathcal{F}\subseteq{}^{\perp_0}I\subseteq{}^{\perp_0}S$. Thus $T$ has the required properties.

The last statement is proven in \cite[Thm.~4.14]{ALS1}.
\end{proof}

\begin{corollary}\label{cor: inj env of simples}
The $\N$-injective envelopes of simple objects $S$ in $\heart{\mathbf{t}}$ are exactly the following complexes in $\N$:
\begin{enumerate}
\item The complexes $\mu_N$ for $N\in \Z$ a critical module in $\F$.  These are precisely the $\N$-injective envelopes of $S\cong F$ with $F$ torsion-free almost torsion with respect to $\Tpair{T}{F}$.
\item The complexes $\mu_N$ for $N\in \Z$ a special module in $\F$.  These are precisely the $\N$-injective envelopes of $S \cong T[-1]$ with $T$ torsion almost torsion-free with respect to $\Tpair{T}{F}$ with a surjective $\F$-cover.
\item The complexes $I[-1]$ for special $I\in \I$.  These are precisely the $\N$-injective envelopes of $S \cong T[-1]$ with $T$ torsion almost torsion-free with respect to $\Tpair{T}{F}$ without a surjective $\F$-cover.
\end{enumerate}
\end{corollary}

We can now describe mutation of cosilting pairs as follows. The proof combines Theorems~\ref{bij_posets} and~\ref{Thm: simple mutation} with Corollary~\ref{cor: inj env of simples} and is left to the reader.

\begin{corollary}
\label{thm:howToDoMutation}
Let $\tpair{u}{v},\tpair{t}{f}\in\tors{A}$ and let 
 $(\mathcal{Z}', \mathcal{I}'), (\Z, \I)$ be the associated cosilting pairs respectively. 
The following statements are equivalent.
\begin{enumerate}
\item $(\mathcal{Z}, \mathcal{I})$ covers $ (\mathcal{Z}', \mathcal{I}') $ in $\Cosiltpair{\Lambda}$.
\item There is a module $T\in\mod{A}$ which is torsion, almost torsion-free with respect to $\tpair{t}{f}$ and there is a unique (critical) module $M$ in $\Z'$ such that $\Hom{\Lambda}(T,M)\neq 0$, satisfying one of the following statements:
\begin{enumerate}
\item $T$ has a surjective $\F$-cover  and $(\Z,\mathcal I)$ is obtained from $( \mathcal{Z}',\, \mathcal{I}' )$ by removing $M$ from $\Z'$ and replacing it with the  unique (special) module $N$ in $\Z$ such that $T\notin\C_{\mu_N}$;
\item $T$ has an injective $\F$-cover and $(\Z,\mathcal I)$ is obtained from $( \mathcal{Z}',\, \mathcal{I}' )$ by removing $M$ from $\Z'$ and replacing it with   the unique (special) module $I$ in $\mathcal{I}$ such that $\Hom{\Lambda}(T,I)\neq 0$.
\end{enumerate}
\item There is a module $V\in\mod{A}$ which is torsion-free, almost torsion with respect to $\tpair{u}{v}$ and there is a unique   (critical) module $M$ in $\Z'$ such that $\Hom{\Lambda}(V,M)\neq 0$, satisfying one of the following statements:
\begin{enumerate}
\item $V$ has a surjective $\F$-cover and $(\Z',\mathcal I')$ is obtained from $( \mathcal{Z},\, \mathcal{I} )$ by removing the unique (special) module  $N$ in $\Z$ such that $V\notin\C_{\mu_N}$ and replacing it with $M$;
\item $V$ has an injective $\F$-cover  and $(\Z',\mathcal I')$ is obtained from $( \mathcal{Z},\, \mathcal{I} )$ by removing the unique (special)  module $I$ in $\mathcal I$ such that $\Hom{\Lambda}(V,I)\neq 0$ and replacing it with $M$.
\end{enumerate}
\end{enumerate}
\end{corollary}

Since $\F$ is a covering class in $\Mod{A}$, there exists an $\F$-cover of any injective cogenerator of $\Mod{A}$.  This $\F$-cover and its kernel will be useful to us in subsequent sections, so we collect together some  properties of these modules. Recall that a module $N$ in $\F$ is said to be \textbf{split-injective in} $\F$ if any monomorphism $N\to F$ with $F\in \F$ splits.

\begin{proposition}\label{prop: neg-is summands}

Let $E$ be an injective cogenerator of $\Mod{A}$ and consider the exact sequence \[0\to C_1 \overset{}{\to} C_0 \overset{f}{\to} E\] where $f$ is an $\F$-cover of $E$.  Then 
the following statements hold.
\begin{enumerate}
\item A module $N\in \F$ is critical in $\F$ if and only if it  is neg-isolated and split-injective in $\F$.
\item The module $C_0$ is split-injective in $\F$, and the module $C_1$ does not have any non-zero direct summands that are split-injective in $\F$.
\item The set $\Prod{C_0\oplus C_1}$ coincides with $\Prod{\Z}$.
\item The direct summands of $C_0$ that are neg-isolated in $\F$ are exactly the critical modules in $\F$.
\item The direct summands of $C_1$ that are neg-isolated in $\F$ are exactly the special modules in $\F$.
\end{enumerate}
\end{proposition}
\begin{proof}
To simplify terminology, we will assume throughout this proof that all terms (split-injective, critical, special, neg-isolated) are being used relative to $\F$.

(1) This is a special case of \cite[Prop.~5.16]{AHL}.

(2) The first statement is \cite[Lem.~3.3]{BreazZemlicka:18},  the second is a consequence of the minimality of $f$.  

(3) We know from \cite[Thm.~3.5]{BreazZemlicka:18} that $C_0\oplus C_1$ is a cosilting module with cosilting class $\F$.   Hence $\Prod{C_0\oplus C_1} = \F\cap\F^{\perp_1} = \Prod{\Z}$.

(4),(5)
Since $C_0$ is split-injective by (2), so is any direct summand of $C_0$.  It follows from (1) that any neg-isolated summand of $C_0$ must be critical.  Similarly, any neg-isolated summand of $C_1$ cannot be critical by  (2), but it is contained in $\F\cap\F^{\perp_1}$ by (3), so it must be special by Remark~\ref{Rem: critical or special}.

Now let $N$ be  critical or special.  Since $N\in\Z\subseteq \Prod{C_0\oplus C_1}$, we must have that $N$ is a direct summand of $C_0$ or $C_1$ by properties of neg-isolated modules 
(see \cite[Prop.~9.29]{Prest:88}).  We have already seen that  the former case applies when $N$ is critical and  the latter case when $N$ is special.
\end{proof}

Now we transfer these properties to the derived category. 

\begin{lemma}\label{lem: summands of products}
Suppose $\mu\in\N$ is the $\N$-injective envelope of a simple in $\heart{\mathbf{t}}$ and that $\mu$ is a direct summand of $\prod_{i\in J} \mu_i$ for some family $\{\mu_i\}_{i\in J}$ of complexes in $\Prod{\N}$.  Then $\mu$ is a direct summand of $\mu_i$ for some $i\in J$.
\end{lemma}
\begin{proof}
Applying $\H{\mathbf{t}}$ to the split monomorphism $\mu\to\prod_{i\in J} \mu_i$ and recalling from Theorem \ref{Thm: cosilting and Zg} that there is an additive equivalence $\H{\mathbf{t}}\colon \Prod{\N} \to \Inj{\heart{\mathbf{t}}}$, we obtain a split monomorphism $\H{\mathbf{t}}(\mu)\to\prod_{i\in J} \H{\mathbf{t}}(\mu_i)$ where $\H{\mathbf{t}}(\mu)$ is the injective envelope of a simple.  A straight-forward argument using properties of injective envelopes yields that $\H{\mathbf{t}}(\mu)$ is a direct summand of $\H{\mathbf{t}}(\mu_i)$ for some $i\in J$.  Applying the additive equivalence from Theorem \ref{Thm: cosilting and Zg} again, we obtain the required conclusion.
\end{proof}

\begin{definition}
Let $\mu\in \N$. \begin{enumerate}
\item If $\mu =\mu_N$ for $N\in \Z$ critical in $\F$, then we say that $\mu$ is \textbf{critical} in $\N$.
\item If $\mu =\mu_N$ for $N\in\Z$ special or $\mu=I[-1]$ for $I\in\I$ special, then we say that $\mu$ is \textbf{special} in $\N$.
\item If $\mu$ is special or critical in $\N$, then we say that $\mu$ is \textbf{neg-isolated} in $\N$.
\end{enumerate}
\end{definition}

\begin{remark}
The neg-isolated objects $\N$ are exactly the $\N$-injective envelopes of simples in $\heart{\mathbf{t}}$ by Corollary \ref{cor: inj env of simples}.
\end{remark}

\begin{proposition}\label{prop: inj env of simples mu_1, mu_0}
Let $E,C_0,C_1,f$ be as in Proposition~\ref{prop: neg-is summands}, and consider $\mu_0 := \mu_{C_0}$.  Let \[\mu_1 \to \mu_0 \overset{f'}{\to} E \to \mu_1[1]\] be the triangle completed from the morphism $f' \colon \mu_0 \to E$ induced by the $\F$-cover $f$.  Then the following statements hold.
\begin{enumerate}
\item The set $\Prod{\mu_0 \oplus \mu_1}$ coincides with $\Prod{\N}$.
\item The direct summands of $\mu_0$ that are neg-isolated in $\N$ are exactly the critical elements of $\N$.
\item The direct summands of $\mu_1$ that are neg-isolated in $\N$ are exactly the special elements of $\N$.
\end{enumerate}
\end{proposition}
\begin{proof}
(1) In the proof of \cite[Prop.~3.5]{Angeleri:18} it is shown that $\sigma:=\mu_0\oplus\mu_1$ satisfies
$\C_\sigma=\C_{\mu_1}=\Cogen{C_0}=\Cogen{C_1\oplus C_0}=\F$. In other words, $\sigma$ is the cosilting complex associated to $\N$ under the bijection in Theorem~\ref{bij_cosilt mod cpx}, and so we have $\Prod{\sigma}=\Prod{\N}$.

(2),(3) Let $\mu$ be neg-isolated in $\N$. Since $\mu \in \Prod{\mu_0\oplus \mu_1}$, it follows that $\mu$ is a direct summand of $\mu_0$ or $\mu_1$ by Lemma \ref{lem: summands of products}.  Note that, by considering the long exact sequence of cohomologies associated to the triangle $\mu_1 \to \mu_0 \to E \to \mu_1[1]$, we have that $\H{}(\mu_1)\cong C_1$.  It follows that, if $\mu$ is a direct summand of $\mu_i$ for $i=0,1$, then $\H{}(\mu)$ is a direct summand of $C_i$.

Now, if $\mu$ is a neg-isolated summand of $\mu_0$, then it cannot be isomorphic to $I[-1]$ for some $I\in\I$.  Indeed, $\mu_0$ is the minimal injective copresentation of $C_0$ by construction, so it cannot contain a direct summand of the form $I[-1]$ or it would contradict minimality.
 It then follows from Proposition~\ref{prop: neg-is summands} and the previous paragraph that $\mu$ is a neg-isolated summand of $\mu_0$ if and only if   $\mu=\mu_N$ for some critical module $N$ in $\F$, proving statement (2).
  
Statement (3) is proven similarly  by using that every neg-isolated element of $\N$ is   critical or special by definition, and it cannot be both by  Proposition \ref{prop: neg-is summands}.
\end{proof}

We saw above that  the complexes $\mu_0$ and $\mu_1$ each determine the torsion-free class $\F$.  We record this fact here.

\begin{corollary}\label{cor: tf class mu}
If $\mu_0$ and $\mu_1$ are as in Proposition \ref{prop: inj env of simples mu_1, mu_0}, then $\F =\Ctf{\mu_1} = \Cogen{\H{}(\mu_0)}$.
\end{corollary}

\section{Wide intervals and closed rigid subsets}\label{Sec: wide closed}

The main aim of this section is to prove the following theorem.

\begin{theorem}\label{thm: wide closed bijection}
There is a bijection $\Phi \colon \WideInt{A} \to \ClRigid{A}$ given by $\Phi([\mathbf{u}, \mathbf{t}]) = \N_{\mathbf{u}} \cap \N_{\mathbf{t}}$ where \[\WideInt{A} := \{[\mathbf{u}, \mathbf{t}]  \mid \tpair{u}{v}, \tpair{t}{f} \in \tors{A} \text{ such that } \mathbf{t}\cap\mathbf{v} \text{ is wide}\}\] and \[\ClRigid{A} := \{\M \in \Rigid{A} \mid \M \text{ closed in } \Zg{\Der{A}}\}.\]
\end{theorem}

Since $\N_{\mathbf{u}}$ and $\N_{\mathbf{t}}$ are closed subsets of $\Zg{\Der{A}}$ by \cite[Prop.~3.2]{ALS1}, the assignment is well-defined.  The next results will allow us to define the inverse assignment.

\begin{lemma}\label{lem: u_M and t_M}
The following statements hold for $\M \in \Rigid{A}$.
\begin{enumerate}
\item There exists a torsion pair of finite type $u_\M := ({}^{\perp_0}\Ctf{\M}, \Ctf{\M}) \in \Cosilt{A}$. 
\item There is an inclusion $\Cogen{\H{}(\M)} \subseteq  \Ctf{\M}$.
\item There exists a torsion pair $t_\M := ({}^{\perp_0}\H{}(\M), \Cogen{\H{}(\M)})$ in $\Mod{A}$.
\end{enumerate}
\end{lemma}
\begin{proof} 
(1) Let $\sigma = \prod_{\mu\in \M} \mu$.  Then $\sigma$ is pure-injective and so $\Ctf{\sigma} = \Ctf{\M}$ is a torsion-free class closed under directed colimits by Lemma~\ref{Csigma}.

(2) Since $\M$ is rigid, we have that $\H{}(\mu') \in \Ctf{\M}$ for all $\mu'\in\M$ by Lemma~\ref{Csigma}(1). It follows from Lemma~\ref{Csigma}(3) that $\Cogen{\H{}(\M)} \subseteq \Ctf{\M}$.  

(3) The inclusion in (2) implies by Lemma~\ref{Csigma} that $\Cogen{\H{}(\M)}$ is contained in ${}^{\perp_1}\H{}(\M)$ and so $t_M$ is a torsion pair by \cite[Rmk.~3.4]{Angeleri:18}.
\end{proof}

\begin{remark}\label{rem: tp coincide if max}(Theorem~\ref{bij_cosilt mod cpx}, \cite[Thm.~2.17 and Def.~2.14]{ALS1})
We have that $u_\M = t_\M$ if and only if $\M \in \CosiltZg{A}$.
\end{remark}

Since we have torsion pairs in $\Mod{A}$, we can consider their HRS-tilts at the standard t-structure. 

\begin{lemma}\label{lem: t-str from M}
Let $\M\in\Rigid{A}$ and consider the right HRS-tilts $\ststr{t^-}$ and $\ststr{u^-}$ of $\ststr{}$ at the torsion pairs $t:=t_\M$ and $u:=u_\M$ defined in Lemma \ref{lem: u_M and t_M}.  The following statements hold.
\begin{enumerate}
\item The aisle of $\ststr{t^-}$ has the form \[\X_{t^-} = {}^{\perp_{\leq0}}\M \cap \D^{\leq 0} = {}^{\perp_{0}}\M \cap \D^{\leq 0}.\]
\item The coaisle of $\ststr{u^-}$ has the form \[\Y_{u^-} = {}^{\perp_{>0}}\M \cap \D^{\geq 0} = {}^{\perp_{1}}\M \cap \D^{\geq 0}.\]
\item The t-structure $\ststr{t^-}$ is the right HRS-tilt of $\ststr{u^-}$ at the hereditary torsion pair \[s = ({}^{\perp_0}\H{u^-}(\M), \Cogen{\H{u^-}(\M)})\] in $\heart{u^-}$.
\end{enumerate}
\end{lemma}
\begin{proof}
To avoid confusion with where orthogonal classes are defined (e.g. in $\Der{A}$, in $\Mod{A}$ or in some heart), let us denote the torsion pairs $t$ and $u$ in $\Mod{A}$ by $(\T_\M, \F_\M)$ and $(\U_\M, \V_\M)$ respectively, and the torsion pair $s$ in $\heart{u^-}$ by $(\S_\M, \R_\M)$.  When we use the ``$\perp_I$" notation in this proof, we will always mean taking the orthogonal in $\Der{A}$. 

(1) Recall from Definition~\ref{HRS} that $\X_{t^-} = \{ X  \in \D^{\leq0} \mid \H{}(X ) \in \T_\M\}$.	 
For degree reasons, we have that $ {}^{\perp_{\leq 0}}\M \cap \D^{\leq0} =  {}^{\perp_{ 0}}\M \cap \D^{\leq0}$.
So, we have to show that a complex $X \in\D^{\leq0}$ satisfies  $H^0(X )\in \T_\M$  if and only if $\Hom{\Der{A}}(X ,\mu)=0$ for all $\mu\in\M$.
Without loss of generality, we can assume that $X $ is a complex with zeroes in positive degrees.  Since the objects in $\M$ are isomorphic to complexes of injective modules that are non-zero only in degrees 0 and 1, any morphism $X  \to \mu$ is a map of chain complexes 
	\[ \xymatrix{\dots \ar[r]  & X^{-1} \ar[r]^-{d^{-1}} \ar[d]  & X^0 \ar[r] \ar[d]^{g} & 0 \ar[r] \ar[d] & 0 \ar[r] \ar[d] & \dots \\
				\dots \ar[r]  & 0 \ar[r] 						 & E_0 \ar[r]^{\mu} & E_1 \ar[r] & 0 \ar[r] & \dots }\]
up to homotopy.  Since there are no non-zero homotopies,  this morphism in $\Der{A}$ is zero if and only if $g = 0$.  As $g \circ d^{-1} = 0$,  there exists some $h \colon \Coker{d^{-1}} \cong \H{}(X ) \to E_0$ such that $g = h \circ \cokermor{d^{-1}}$.  Moreover $\mu \circ h \circ \cokermor{d^{-1}} = \mu \circ g = 0$ and so $\mu \circ h = 0$.  Then there exists $l \colon \H{}(X ) \to \Ker{\mu} \cong \H{}(\mu)$ with $h = \kermor{\mu}\circ l$.  If we suppose that $\H{}(X ) \in  \T_\M$, it follows that $l=0$ by definition of $\T_\M$, thus $g = 0$, and so $\Hom{\Der{A}}(X ,\mu)=0$ as required. The reverse implication is shown similarly.

(2) We recall from Definition~\ref{HRS} that $\Y_{u^-} = \{ X  \in  \D^{\geq0} \mid \H{}(X) \in \V_\M \}$, and we note that  ${}^{\perp_{> 0}}\M \cap \D^{\geq0}={}^{\perp_{1}}\M \cap \D^{\geq0}$
 for degree reasons.
So, we have to show that a complex $X \in\D^{\geq0}$ satisfies  $H^0(X )\in \V_\M$  if and only if $\Hom{\Der{A}}(X ,\mu[1])=0$ for all $\mu\in\M$. We  assume that $X $ is a complex with zeroes in negative degrees. 

Consider a morphism
\[ \xymatrix{\dots \ar[r]  & 0 \ar[r] \ar[d]  & 0 \ar[r] \ar[d] & X^0 \ar[r]^{d^0} \ar[d]^{g} & X^1 \ar[r] \ar[d] & \dots \\
				\dots \ar[r]  & 0 \ar[r] 						 & E_0 \ar[r]^{\mu} & E_1 \ar[r] & 0 \ar[r] & \dots }\]
Let $g' := g\circ \kermor{d^0} \colon \Ker{d^0} \cong \H{}(X ) \to E_1$.  If we suppose that $\H{}(X ) \in \V_\M$, there exists  $h_0' \colon \H{}(X ) \to E_0$ such that $\mu\circ h_0' = g'$ by definition of $\V_\M$.  Using that $E_0$ is injective, we obtain a morphism $h_0 \colon X^0 \to E_0$ such that $h_0 \circ \kermor{d^0} = h_0'$.  Now consider the short exact sequence 
	\[0 \to \H{}(X ) \overset{\kermor{d^0}}{\longrightarrow} X^0  \overset{\pi}{\longrightarrow} \Im{d^0} \to 0.\]
Note that $\mu \circ h_0 \circ \kermor{d^0} = \mu \circ h_0' = g' = g\circ\kermor{d^0}$, so $0 = (g - \mu \circ h_0)\circ\kermor{d^0}$.  Therefore there exists a morphism $h_1' \colon \Im{d^0} \to E_1$ such that $h_1'\circ \pi = g-\mu\circ h_0$.  Since $E_1$ is injective,  there exists $h_1 \colon X^1 \to E_1$ such that $h_1\circ \iota = h_1'$ where $\iota \colon \Im{d^0}\to X_1$ is the inclusion.  A routine computation shows that $h_0$ and $h_1$ are the only two non-trivial terms in a homotopy witnessing that $g=0$ in $\Der{A}$, and so $\Hom{\Der{A}}(X ,\mu[1])=0$ as required. The reverse implication is shown similarly.

(3)   By Lemma \ref{lem: u_M and t_M}(2) we have that $\F_\M \subseteq \V_\M$ and so it follows from \cite[Prop.~7.5]{ALSV} that $\ststr{t^-}$ is a right HRS-tilt of $\ststr{u^-}$ at a torsion pair 
$(\S, \R)$ in $\heart{u^-}$ with $\S=\T_\M\cap\V_\M$. We have to show that $\S=\S_\M$. Note that for a module $Y\in\Mod{A}$ which belongs to $\heart{u^-}$ and for an element $\mu\in\M$ we have that
$$\Hom{A}(Y,\H{}(\mu))\cong \Hom{\Der{A}}(Y, \mu) \cong \Hom{\heart{u^-}}(Y, \H{u^-}(\mu)),$$ thus $Y$ lies in $\T_\M$ if and only if  it lies in $\S_\M$. This shows the inclusion $\S\subseteq\S_\M$. For the reverse inclusion, let $Y\in\S_\M$ and consider the short exact sequence $0 \to V \to Y \to U[-1] \to 0$ given by the torsion pair $(\V_\M, \U_\M[-1])$ in $\heart{u^-}$.  Since $\S_\M$ is a torsion class, we have that $U[-1]\in \S_\M$.  By definition of $\S_\M$, this means that \[0 = \Hom{\heart{u^-}}(U[-1], \H{{u^-}}(\M)) \cong \Hom{\Der{A}}(U[-1], \M) \cong \Hom{\Der{A}}(U, \M[1]).\]  It follows that $U\in \U_\M \cap \V_\M = 0$ and so $Y\cong V \in \V_\M$. Since $Y$ is a module, we have already observed that $Y\in\T_\M$, and so $Y\in\T_\M\cap\V_\M=\S$ as desired.
\end{proof}

Next we establish when $t_\M$ is contained in $\Cosilt{A}$. Given a class of indecomposable objects $\X$ in an additive category with products, we  say that $\X$ is \textbf{product-rigid} if $\Ind{\Prod{\X}} = \X$.

\begin{remark}\label{closedisproductrigid}  Every closed set $\M$ in $\Zg{\Der{A}}$ or $\Zg{A}$ is product-rigid. Indeed, by the definition of the Ziegler topology, if $\M$ is a closed set, then there exists a definable subcategory $\mathcal{D}$  such that $\M$ consists of the indecomposable pure-injective objects in  $\mathcal{D}$. In that case $\M = \Ind{\Prod{\M}}$ because $\mathcal{D}$ is closed under products and summands.
\end{remark}

\begin{proposition}\label{prop: closed sets}
 The following statements are equivalent for $\M\in\Rigid{A}$. 
\begin{enumerate}
\item The set $\M$ is closed in $\Zg{\Der{A}}$.
\item The set $\M$ is product-rigid, and $t_\M = ({}^{\perp_0}\H{}(\M), \Cogen{\H{}(\M)}$ belongs to  $\Cosilt{A}$.
\item The set $\H{}(\M)$ is closed in $\Zg{A}$.
\end{enumerate}
\end{proposition} 
\begin{proof}
To simplify notation, we will fix the following throughout the proof: $u:=u_\M$, $t :=t_\M$ and we will denote the restriction of $t$ to $\mod{A}$ by $\tpair{t}{f}$. 

We have just seen in Remark~\ref{closedisproductrigid} that every closed set $\M$  is product-rigid. Recall from Theorem \ref{Thm: cosilting and Zg} that $\H{u^-}\colon \Prod{\N_\u} \to  \Inj{\heart{u^-}}$ is an additive equivalence that commutes with direct summands and products. Therefore   $\H{u^-}(\M)$ is product-rigid as well, and so is  $\H{}(\M)$ by a standard argument. Moreover, if $\H{}(\M)$ is product-rigid, then so is $\M$.  Indeed, suppose $\H{}(\M)$ is product-rigid and let $X\in \Ind{\Prod{\M}}$. Then either $X \cong I[-1]$ for some indecomposable injective $I$ or $X\cong \mu_M$ for some $M\in \H{}(\Ind{\Prod{\M}})$ by \cite[Lem.~4.4(2)]{ALS1}. In the former case, it follows directly from Corollary~\ref{cor: inj env of simples} that $X$ is neg-isolated in any completion of $\M$ to a maximal rigid set (which exists by \cite[Lem.~3.7]{ALS1}) and so $X \in \M$ by Lemma \ref{lem: summands of products}. In the latter case, we observe that it follows from \cite[Lem.~4.4]{ALS1} that $\H{}(\Ind{\Prod{\M}}) = \Ind{\Prod{\H{}(\M)}}$ which is equal to $\H{}(\M)$ by assumption. Thus $M\in \H{}(\M)$ and so $\mu_M\in \M$, again, by \cite[Lem.~4.4]{ALS1}(1).

We employ the homeomorphism  $\H{u^-}\colon \N_\u \to \Sp{\heart{u^-}}$ from Theorem \ref{Thm: cosilting and Zg} to verify that (1) and (2) are equivalent.  Indeed, $\M$ is closed in $\Zg{\Der{A}}$ if and only if it is closed in the subspace topology on the closed set $\N_\u$, if and only if $\H{u^-}(\M)$ is closed in $\Sp{\heart{u^-}}$. By definition of the topology on $\Sp{\heart{u^-}}$ and because $\H{u^-}(\M)$ is product-rigid, the latter means that the hereditary torsion pair $s = ({}^{\perp_0}\H{u^-}(\M), \Cogen{\H{u^-}(\M)})$ in $\heart{u^-}$ is of finite type, see \cite[Thm.~A6]{ALS1}. Now we use \cite[Prop.~4.5]{ALSV} which states that a torsion pair  in the heart $\heart{}$ of  a t-structure $\tstr{}$ is of finite type if and only if the corresponding HRS-tilt of $\tstr{}$   is a cosilting t-structure (such a torsion pair is a cosilting torsion pair in the terminology of \cite[Def.~3.9]{ALSV}). We apply this statement both to the t-structure $\ststr{u^-}$ with  torsion pair $s$ in  $\heart{u^-}$ and  to the standard t-structure $\ststr{}$ with  torsion pair $t$ in  $\Mod{A}$. Since $\ststr{t^-}$ is a right HRS-tilt of $\ststr{u^-}$ at $s$,  we conclude that $s$ is of finite type in $\heart{u^-}$ if and only if  $t$ is a torsion pair of finite type in $\Mod{A}$.

Now we show that (2) implies (3).  If (2) holds, then $\H{}(\M)$ is contained in $\Z_{\mathbf{t}}$ because $\Z_{\mathbf{t}}$ consists of Ext-injectives in $\Cogen{\H{}(\M)}$ and Lemma \ref{lem: u_M and t_M}(2) implies that the inclusion $\Cogen{\H{}(\M)}\subseteq {}^{\perp_1}\H{}(\M)$ holds.  In the case where $\Cogen{\H{}(\M)}$ is a cotilting class, we have that $\M = \H{}(\M)$ and $\N_{\mathbf{t}} = \Z_{\mathbf{t}}$ in $\Der{A}$.  We can then conclude that $\M = \H{}(\M)$ is a closed set in $\Zg{\Der{A}}$ by the first part of the proof and hence also in $\Zg{A}$.  In general $\Cogen{\H{}(\M)}$ is a cotilting class in $\Mod{\bar{A}}$ and so we have just seen that $\H{}(\M)$ is a closed set in $\Zg{\bar{A}}$. As $\Zg{\bar{A}}$ is homeomorphic to the closed set $\Zg{A}\cap \Mod{\bar{A}}$ in $\Zg{A}$ with the subspace topology (by \cite[Thm.~1, Cor.~9]{Prest:96}), we can conclude that $\H{}(\M)$ is also a closed set in $\Zg{A}$.

Finally, if (3) holds, then $\Cogen{\H{}(\M)}$ is definable, as it can be regarded as the class of subobjects of some definable subcategory, see \cite[Prop.~3.4.15]{Prest:09}, namely the definable subcategory corresponding to the closed set $\H{}(\M)$, see \cite[Prop.~5.4.53]{Prest:09}. Definable subcategories are closed under direct limits, so $t$ is a  torsion pair of finite type. We have already seen that $\M$ is product rigid, so (2) is verified.
\end{proof}

In Lemma \ref{lem: u_M and t_M} we have defined two torsion pairs $u_\M$ and $t_\M$ in $\Mod{A}$.  By restricting them to $\mod{A}$, we can produce a pair of torsion pairs in $\tors{A}$, which we will denote by $\tpair{u_\M}{v_\M}$ and $\tpair{t_\M}{f_\M}$. By Lemma \ref{lem: u_M and t_M}(2), we have that $\tpair{u_\M}{v_\M} \leq \tpair{t_\M}{f_\M}$.  

\begin{corollary}\label{cor: Phi surjective, Psi injective}
The assignment \[\Psi \colon \ClRigid{A} \to \WideInt{A}, \, \M \mapsto [\mathbf{u}_\M, \mathbf{t}_\M]\]  
is well-defined, and $\Phi\circ\Psi$ is the identity on $\ClRigid{A}$.
\end{corollary}

\begin{proof}
We fix $\M \in \ClRigid{A}$ and consider the torsion pairs $u:=u_\M$, $t :=t_\M$ in $\Cosilt{A}$  together  with the restricted torsion pairs $\tpair{u}{v}:=\tpair{u_\M}{v_\M}$ and 
$\tpair{t}{f}:=\tpair{t_\M}{f_\M}$   in $\mod{A}$.  
By Lemma \ref{lem: t-str from M}, we know that $\ststr{t^-}$ is the right HRS-tilt of $\ststr{u^-}$ at the hereditary torsion pair $s = ({}^{\perp_0}\H{u^-}(\M), \Cogen{\H{u^-}(\M)})$.
Combining Theorems~\ref{Thm: serre hereditary mutat} and~\ref{Thm: wide} we see that $[\mathbf{u}, \mathbf{t}]$ is a wide interval. This proves that
the assignment $\Psi$ is well-defined. 
To prove that $\Phi\circ\Psi$ is the identity, we set  $\E:= \Phi([\mathbf{u}, \mathbf{t}])=\N_\mathbf{u} \cap \N_{\mathbf{t}}$ and prove that $\M = \E$.   By Theorem~\ref{Thm: serre hereditary mutat} we have that $\Cogen{\H{u^-}(\M)} = \Cogen{\H{u^-}(\E)}$, and thus $\Prod{\H{u^-}(\M)} = \Prod{\H{u^-}(\E)}$.  Since $\H{u^-}\colon \Prod{\N_\mathbf{u}} \to  \Inj{\heart{u^-}} $ is an additive equivalence that commutes with direct summands and products (see Theorem \ref{Thm: cosilting and Zg}), it follows that $\Prod{\M} = \Prod{\E}$ and hence $\M = \Ind{\Prod{\M}} = \Ind{\Prod{\E}} = \E$, where the first and last equality hold because $\M$ and $\E $ are closed sets of $\Zg{\Der{A}}$.  
\end{proof}

In order to prove Theorem \ref{thm: wide closed bijection}, that is, that $\Phi$ and $\Psi$ are mutually inverse bijections, we first need to prove some technical lemmas.

\begin{definition}
Let $\mu$ be a pure-injective object in $\Der{A}$.  We say that a subset $\M$ of $\Zg{\Der{A}}$ is the \textbf{support of $\mu$} if $\M$ is the smallest closed subset of $\Zg{\Der{A}}$ such that $\mu \in \Prod{\M}$.  We use the notation $\Supp{\mu}:=\M$.
\end{definition}

\begin{lemma}\label{lem: support mu_1}
Let $\N\in\CosiltZg{A}$ and consider the triangle $\mu_1 \to \mu_0 \to E \to \mu_1[1]$ defined in Proposition \ref{prop: inj env of simples mu_1, mu_0}.  Then the support $\Supp{\mu_1}$ of $\mu_1$ coincides with the closed set obtained by removing all critical elements of $\N$ that are $\N$-injective envelopes of finitely presented simples.
\end{lemma}
\begin{proof}
Since $\mu_1\in\Prod{\N}$, we have that $\Supp{\mu_1}\subseteq \N$. Thus $\Supp{\mu_1}$ will coincide with the smallest closed subset $\M$ in the subspace topology on $\N$ such that $\mu_1\in\Prod{\M}$.

Let $u = \Tpair{U}{V}\in\Cosilt{A}$ be the  torsion pair of finite type associated to $\N$ under the bijections in Theorem \ref{bij_posets}.
 
Fix the notation $E_1 := \H{u^-}(\mu_1)\in\Inj{\heart{u^-}}$.  We will identify the smallest hereditary torsion-free class of finite type in $\heart{u^-}$ containing $E_1$ and the corresponding closed set $\L$ of $\Sp{\heart{u^-}}$ is the smallest closed set such that $E_1\in\Prod{\L}$. Using the additive equivalence $\H{u^-}\colon \Prod{\N} \to \H{u^-}(\Prod{\N}) = \Inj{\heart{u^-}}$ and the homeomorphism $\H{u^-}\colon \N \to \Sp{\heart{u^-}}$ from Theorem \ref{Thm: cosilting and Zg}, we may then conclude that $\Supp{\mu_1}$ is the inverse image of $\L$ under $\H{u^-}$.

According to \cite[Thm.~3.8]{Herzog:97}, the smallest hereditary torsion-free class of finite type in $\heart{u^-}$ containing $E_1$ is given by $(\varinjlim{\S}, \S^{\perp_0})$ where $\S = {}^{\perp_0}E_1 \cap \fp{\mathcal H}_{u^-}$.  Now consider any $U[-1] \in \U[-1]$ and note that 
	\[\Hom{\heart{u^-}}(U[-1], E_1) \cong \Hom{\Der{A}}(U[-1], \mu_1) \cong \Hom{\Der{A}}(U, \mu_1[1]).\]  
So $U[-1]\in {}^{\perp_0}E_1$ if and only if $U\in \U \cap \Ctf{\mu_1}$ by Lemma \ref{Csigma}(1).  But since $\Ctf{\mu_1} = \V$ by Corollary \ref{cor: tf class mu}, we have that $U[-1]\in {}^{\perp_0}E_1$ if and only if $U = 0$.  In other words, we have that ${}^{\perp_0}E_1 \subseteq \V$ and hence $\S = {}^{\perp_0}E_1 \cap \fp{\mathcal H}_{u^-} \subseteq \mathbf{v}$.  In particular, we have that $\S$ is a Serre subcategory of $\fp{\mathcal H}_{u^-}$ contained in $\mathbf{v}$.  By \cite[Lem.~1.2.24]{LakingICRA}, there exists some set of finitely presented simples $\Omega\subseteq \mathbf{v}$ such that $\S = \filt{\Omega}$.  Moreover, if $V\in\V$ is simple in $\heart{u^-}$ and $\Hom{\heart{u^-}}(V, E_1) \neq 0$, then $\mu_1$ has a critical summand, contradicting Proposition \ref{prop: inj env of simples mu_1, mu_0}.  Thus the simple objects contained in ${}^{\perp_0}E_1$ are exactly those contained in $\V$ and, in particular, the set $\Omega$ consists of all of the simple objects in $\mathbf{v}$.

Now we have showed that $\L$ is given by $\S^{\perp_0} \cap \Sp{\heart{u^-}} = \Omega^{\perp_0} \cap \Sp{\heart{u^-}}$ and so $\L$ is obtained by removing all of the injective envelopes of finitely presented simples $V\in\V$ in from the set $\Sp{\heart{u^-}}$.  The statement of the lemma follows from Corollary \ref{cor: inj env of simples}.
\end{proof}

\begin{lemma}\label{lem: tf classes closed subset}
Let $\N\in\CosiltZg{A}$.  \begin{enumerate}
\item If $\M \subseteq \N$ is closed and $\N\setminus \M$ consists of critical objects that are $\N$-injective envelopes of finitely presented simples, then we have the equality $\Ctf{\N} = \Ctf{\M}$.
\item If $\M \subseteq \N$ is closed and $\N\setminus \M$ consists of special objects, then we have the equality $\Cogen{\H{}(\N)} = \Cogen{\H{}(\M)}$.
\end{enumerate}
\end{lemma}
\begin{proof}
(1) By Lemma \ref{lem: support mu_1}, we have that $\Supp{\mu_1} \subseteq \M$ so $\Ctf{\M} \subseteq \Ctf{\Supp{\mu_1}}= \Ctf{\Prod{\Supp{\mu_1}}}\subseteq \Ctf{{\mu_1}}=\Ctf{\N} \subseteq \Ctf{\M}$ where the last equality is given by Corollary \ref{cor: tf class mu}.

(2) Let $\N_0$ denote the set of critical elements of $\N$ and note that $\N_0 \subseteq \M$.  Then $\Z_0 := \H{}(\N_0)$ is the set of critical elements of $\Z=\H{}(\N)$ by definition. Then, by \cite[Prop.~5.16]{AHL}, we have 
\[\Cogen{\Z} = \Cogen{\Z_0} \subseteq \Cogen{\H{}(\M)} \subseteq \Cogen{\Z}.\]
\end{proof}

\begin{proof}[Proof of Theorem \ref{thm: wide closed bijection}]
We already have the following well-defined assignments 
	\begin{equation*}\begin{split}\Phi \colon \WideInt{A} &\to \ClRigid{A} \\  [\mathbf{u},\mathbf{t}]&\mapsto \closed{\mathbf{u}} \cap \closed{\mathbf{t}}  \end{split}\end{equation*}
		\begin{equation*}\begin{split} \Psi \colon \ClRigid{A} &\to \WideInt{A} \\  \M &\mapsto [\mathbf{u}_\M, \mathbf{t}_\M] \end{split}\end{equation*}
such that $\Phi\circ\Psi $ is the identity on $\ClRigid{A}$ by Corollary \ref{cor: Phi surjective, Psi injective}.  We conclude by showing that $\Psi \circ \Phi$ is the identity on $\WideInt{A}$. Let $[\mathbf{u},\mathbf{t}] \in \WideInt{A}$ and denote by $\M := \Phi([\mathbf{u}, \mathbf{t}]) = \closed{\mathbf{u}}\cap \closed{\mathbf{t}}$. By Theorem~\ref{Thm: simple mutation} and Corollary~\ref{cor: inj env of simples} we have that $\closed{\mathbf{u}}\setminus \M$ consists of critical $\closed{\mathbf{u}}$-injective envelopes of finitely presented simples and 
$\closed{\mathbf{t}}\setminus \M$ consists of special elements of $\closed{\mathbf{u}}$.
It then follows from Lemma \ref{lem: tf classes closed subset} that $\Ctf{\closed{\mathbf{u}}} = \Ctf{\M}$ and thus $\mathbf{u} = \mathbf{u}_\M$, and furthermore, $\Cogen{\H{}(\closed{\mathbf{t}})} = \Cogen{\H{}(\M)}$ and hence $\mathbf{t} = \mathbf{t}_\M$. In other words, we have shown that $\Psi \circ \Phi([\mathbf{u},\mathbf{t}]) = [\mathbf{u},\mathbf{t}]$.
\end{proof}

\begin{corollary}\label{Cor: completions}
Let $\M\in\ClRigid{A}$. The interval $[\mathbf{u}_\M, \mathbf{t}_\M]$  consists of the   torsion pairs in $\tors{A}$ whose associated maximal rigid set $\N_{}$  contains the closed rigid set $\M$.
\end{corollary}
\begin{proof}
If $\closed{}$ contains $\M$, then we have inclusions of torsion-free classes $$\Cogen{\H{}(\M)} \subseteq \Cogen{\H{}(\closed{})} = \Ctf{\closed{}} \subseteq \Ctf{\M}$$ and  the statement is clear by definition of the torsion pairs $u_\M$ and $t_\M$.

For the converse,  assume that $\N$ is associated to a torsion pair in the interval $[\mathbf{u}, \mathbf{t}]:=[\mathbf{u}_\M, \mathbf{t}_\M]$.  By  Theorem \ref{bij_posets} we have that $\closed{\mathbf{u}} \leq \closed{} \leq \closed{\mathbf{t}}$. It suffices to show that $\M\subseteq \closed{\mathbf{u}}$ and $\M \subseteq \closed{\mathbf{t}}$ since we would then have that $\Hom{\Der{A}}(\M, \closed{}[1])=0$ by $\closed{\mathbf{u}} \leq \closed{}$ and $\Hom{\Der{A}}(\closed{}, \M[1])=0$ by $\closed{} \leq \closed{\mathbf{t}}$, and so $\M \subseteq \closed{}$ holds by maximality of $\closed{}$.

We show that $\M$ is contained both in $\Y_{u^-}\cap{\Y_{u^-}}^{\perp_1}$ and $\Y_{t^-}\cap {\Y_{t^-}}^{\perp_1}$. This will imply our claim, because we know from \cite[Lem.~4.5(iii)]{PsaroudakisVitoria:18} that these intersections coincide with $\Prod{\closed{\mathbf{u}}}$ and $\Prod{\closed{\mathbf{t}}}$, respectively, and the closed sets ${\closed{\mathbf{u}}}$ and ${\closed{\mathbf{t}}}$ are  product-rigid.
First we note that $\Y_{t^-}\subseteq \Y_{u^-}$ and so we only need to show that $\M\subseteq \Y_{t^-}$ and that $\M \subseteq {\Y_{u^-}}^{\perp_1}$.  But this follows directly from Lemma \ref{lem: t-str from M} because $\M\subseteq ({}^{\perp_0}\M \cap \D^{\leq 0})^{\perp_0} = {\X_{t^-}}^{\perp_0} = \Y_{t^-}$ and $\M \subseteq ({}^{\perp_1}\M \cap \D^{\geq 0})^{\perp_1} = {\Y_{u^-}}^{\perp_1}$.
\end{proof}

\section{Topological characterisations of irreducible mutation}\label{Sec: top char}

We have seen in Corollary~\ref{Thm: mutation} that irreducible mutations in $\CosiltZg{A}$ correspond to intervals $[\mathbf{u}, \mathbf{t}]$ in $\tors{A}$ such that $\mathbf{t}$ covers $\mathbf{u}$. Such intervals are wide and so we can apply the results of Section \ref{Sec: wide closed} to characterise irreducible mutations in terms of the corresponding closed subset given by Theorem \ref{thm: wide closed bijection}.

\begin{notation}
Let $\M\in\ClRigid{A}$.  Then we denote the maximal rigid set $\closed{\mathbf{t}_\M}$ by $\M^+$ and the maximal rigid set $\closed{\mathbf{u}_\M}$ by $\M^-$.
\end{notation}

\begin{definition} A closed rigid set $\M$ is called \textbf{almost complete} if is not contained in $\CosiltZg{A}$ and every $\N\in\ClRigid{A}$ that properly contains $\M$ is contained in $\CosiltZg{A}$.
\end{definition}

\begin{theorem}\label{Thm: topological char}
The following statements are equivalent for $\M\in \ClRigid{A}$.
\begin{enumerate}
\item The maximal rigid set $\M^+$ is an irreducible right mutation of $\M^-$.
\item The maximal rigid set $\M^-$ is an irreducible left mutation of $\M^+$.
\item The torsion pair $(\mathbf{t}_\M, \mathbf{f}_\M)$ covers $(\mathbf{u}_\M, \mathbf{v}_\M)$ in $\tors{A}$.
\item There are exactly two ways to complete $\M$ to a maximal rigid set.
\item The set $\M$ is almost complete.
\end{enumerate}
\end{theorem}
\begin{proof}
(1)$\Leftrightarrow$(2) This is by definition.

(2)$\Leftrightarrow$(3) This is Corollary~\ref{Thm: mutation}.

(3)$\Leftrightarrow$(4) This follows from Corollary \ref{Cor: completions}.

(3)$\Rightarrow$(5)  Suppose that $\N \in \ClRigid{A}$ with $\M\subsetneq \N$.  Then $\mathbf{u}_\M \subseteq \mathbf{u}_\N \subseteq \mathbf{t}_\N \subseteq \mathbf{t}_\M$.  Since $\M\neq \N$, we cannot have that $\mathbf{u}_\M = \mathbf{u}_\N$ and $\mathbf{t}_\M = \mathbf{t}_\N$ by Theorem \ref{thm: wide closed bijection}.  As the interval $[\mathbf{u}_\M, \mathbf{t}_\M]$ is simple, it follows that either $\mathbf{u}_\M = \mathbf{u}_\N = \mathbf{t}_\N$ or $\mathbf{u}_\N = \mathbf{t}_\N = \mathbf{t}_\M$.  In both cases we can conclude that $\N$ is maximal rigid because $\mathbf{u}_\N = \mathbf{t}_\N$. See Remark \ref{rem: tp coincide if max}.

(5)$\Rightarrow$(1)  Let $\mathbf{u}:=\mathbf{u}_\M$,  and $\mathbf{t} := \mathbf{t}_\M$. We know that $[\mathbf{u}, \mathbf{t}]$ is a wide interval by Corollary \ref{cor: Phi surjective, Psi injective} and so $\M^+$ is a right mutation of $\M^-$ by Theorem~\ref{Thm: wide}.  By 
Theorem~\ref{Thm: simple mutation}
there exists a set $\Omega$ of (pairwise non-isomorphic) finitely presented simple objects in $\heart{\mathbf{u}}$ such that $\M^-$ is the disjoint union of $\M=\M^+\cap\M^-$ with the set of $\M^-$-injective envelopes of the objects in $\Omega$. Using the homeomorphism $\H{\mathbf{u}}\colon \M^- \to \Sp{\heart{\mathbf{u}}}$ from Theorem \ref{Thm: cosilting and Zg}, we see that  $\Sp{\heart{\mathbf{u}}}$ is the disjoint union of $\H{\mathbf{u}}(\M)$ with the set of injective envelopes $E_\mathbf{u}(S)$ of the simple objects  $S\in\Omega$. By definition of the topology on $\Sp{\heart{\mathbf{u}}}$ every $S\in\Omega$ induces a closed set  $\Sp{\heart{\mathbf{u}}}\setminus\{E_\mathbf{u}(S)\}=\Sp{\heart{\mathbf{u}}}\setminus\open{S}$ in $\Sp{\heart{\mathbf{u}}}$. Now, if we assume that $\M^+$ is not an irreducible  mutation of $\M^-$, then the set $\Omega$ contains more than one simple, and $\Sp{\heart{\mathbf{u}}}\setminus \{E_{\mathbf{u}}(S)\}$  strictly contains $\H{\mathbf{u}}(\M)$.  Using the homeomorphism $\H{\mathbf{u}}\colon \M^{\chl{-}} \to \Sp{\heart{\mathbf{u}}}$, this yields a closed rigid set $\L$ such that $\M \subsetneq \L \subsetneq \M^{\chl{-}}$ and so the condition (5) does not hold.
\end{proof}

Let $\tpair{t}{f} \in \tors{A}$. Recall that we call $\mu\in\N_{\mathbf{t}}$ mutable if there exists $\N = (\N_{\mathbf{t}} \setminus \{\mu\}) \cup \{\nu\}$ in $\CosiltZg{A}$ such that $\N_{\mathbf{t}}$ and $\N$ are related by irreducible mutation.  By Corollary \ref{Thm: mutation}, this is equivalent to $\mu$ being the $\N_{\mathbf{t}}$-injective envelope of a finitely presented simple object in $\heart{\mathbf{t}}$.

\begin{corollary}\label{cor: mutable means exactly two completions}
Let $\N\in\CosiltZg{A}$ and $\mu\in\N$.  Then $\mu$ is mutable if and only if there exists a unique $\N'\in\CosiltZg{A}$ such that $\N'\neq \N$ and $\N\setminus \{\mu\} \subseteq \N'$.
\end{corollary}
\begin{proof}
$(\Rightarrow)$ Let $\M := \N\setminus \{\mu\}$. Since $\mu$ is mutable, there exists $\N' = \M \cup \{\nu\}\in\CosiltZg{A}$ such that $\N'$ is either an irreducible right or irreducible left mutation of $\N$.  We will consider the case where $\N'$ is a right mutation of $\N$; the case of left mutation is similar.  First we note that $\M$ is equal to $\N\cap\N'$ and hence is contained in $\ClRigid{A}$.  By Corollaries \ref{Thm: mutation} and \ref{cor: inj env of simples}, we have that $\mu$ is critical in $\N$ and an $\N$-injective envelope of a finitely presented simple, and $\nu$ is special in $\N'$.  By Lemma \ref{lem: tf classes closed subset}, we have that $\N = \M^-$ and $\N'=\M^+$.  The statement then follows directly from Theorem \ref{Thm: topological char}.

$(\Leftarrow)$ Let $\M := \N\setminus \{\mu\}$. By assumption, we have that $\M$ is equal to $\N \cap \N'$ and so is closed.  Thus $\M$ satisfies condition (4) of Theorem \ref{Thm: topological char} meaning that we must have $\{\M^+, \M^-\} = \{\N, \N'\}$ and $\N$ and $\N'$ are related by irreducible mutation.
\end{proof}

In the previous corollary we are using that, when $\mu\in \N\in \CosiltZg{A}$ is mutable, the set $\M := \N \setminus \{\mu\}$ is closed and so we can employ the theory developed in Section \ref{Sec: wide closed}. In other words, $\mu$ is an isolated point of $\N$ in the Ziegler subspace topology on $\N$. 
\begin{question}
Is it possible to have an element $\mu\in \N\in \CosiltZg{A}$ that is isolated but it is not mutable? This would mean that it would be possible to complete $\M:= \N \setminus \{\mu\}$ to multiple different maximal rigid sets but \emph{not} via irreducible mutation. 
\end{question}

\subsection{The weak isolation property in the derived category}
In this section we will focus on the setting where being isolated in $\N$ is also a sufficient condition for mutability.  

\begin{definition}
Let $\heart{}$ be a locally coherent Grothendieck category.  We say that $\heart{}$ has the \textbf{weak isolation property} if every isolated point $E$ in $\Sp{\heart{}}$ is the injective envelope of a finitely presented simple object in $\heart{}$.

We say that $\N\in\CosiltZg{A}$ has the \textbf{weak isolation property} if $\N=\N_{\mathbf{t}}$ for $\tpair{t}{f}\in\tors{A}$ and $\heart{\mathbf{t}}$ has the weak isolation property.
\end{definition}

\begin{remarks}
\begin{enumerate}
\item A maximal rigid set $\N_{\mathbf{t}}$ has the weak isolation property if and only if every isolated point in $\N_{\mathbf{t}}$ is the $\N_{\mathbf{t}}$-injective envelope of a finitely presented simple in $\heart{\mathbf{t}}$. 

\item Let $\heart{}$ be a locally coherent Grothendieck category.  If $E\in\Sp{\heart{}}$ is the injective envelope of a finitely presented simple object $S$, then $E$ is isolated.  Indeed, following the notation of \cite[App.~A]{ALS1}, we have that $\mathscr{O}(S) = \{E\}$. 
\end{enumerate}
\end{remarks}

The following lemma then follows easily from the definitions and Corollary \ref{Thm: mutation}.

\begin{lemma}\label{lem: weak isol mutability}
The following statements are equivalent for $\N\in\CosiltZg{A}$.
\begin{enumerate}
\item The maximal rigid set $\N$ has the weak isolation property.
\item An object $\mu\in\N$ is mutable if and only if it is isolated in  $\N$.
\end{enumerate}
\end{lemma}

Let us  summarise our findings.
\begin{proposition}
Suppose $\N\in\CosiltZg{A}$ has the weak isolation property. Then following statements are equivalent for $\mu\in\N$ and $\M := \N\setminus\{\mu\}$.
\begin{enumerate}
\item There exists a (unique) $\N' \in\CosiltZg{A}$ such that $\N'\neq \N$ and $\M\subseteq \N'$.
\item The object $\mu$ is isolated in $\N$.
\item The object $\mu$ is mutable in $\N$.
\end{enumerate}
\end{proposition}
\begin{proof}
(1)$\Rightarrow$(2) The subset $\N \cap \N'= \M$ is closed in $\N$ by Theorem \ref{Thm: cosilting and Zg}(1) and $\{\mu\} = \N \setminus \M$.

(2)$\Leftrightarrow$(3) This is Lemma \ref{lem: weak isol mutability}.

(3)$\Rightarrow$(1) This is Corollary \ref{cor: mutable means exactly two completions}.
\end{proof}

\begin{corollary}
Suppose $\N\in\CosiltZg{A}$ has the weak isolation property. Then, for every $\mu\in\N$, either $\mu$ is mutable or $\N$ is the unique maximal rigid subset containing $\N\setminus\{\mu\}$. In the latter case we will call $\mu$ \textbf{immutable}.
\end{corollary}

A locally coherent Grothendieck category $\heart{}$ is said to have the \textbf{isolation property} if every localisation $\heart{}/\T$ with $(\T, \F)$ a hereditary torsion pair of finite type in $\heart{}$ has the weak isolation property. Clearly the isolation property implies the weak isolation property.  

This stronger condition was introduced by Burke in \cite{Burke:01}, generalising a definition used in the model theory of modules (see, for example, \cite{Prest:09}), and he showed that any locally coherent Grothendieck category without superdecomposable injective objects has the isolation property.  Recall that an object $X$ in a category $\C$ is called \textbf{superdecomposable} if it has no indecomposable direct summands.  Using this result, we can deduce the following proposition.

\begin{proposition}\label{prop: superdecomp suff cond tri}
Let $\t\in\tors{A}$. The following statements hold. 
\begin{enumerate}
\item If $\heart{\mathbf{t}}$ does not contain a superdecomposable injective object, then $\N_{\mathbf{t}}$ has the weak isolation property.
\item If $\Prod{\N_{\mathbf{t}}}$ does not contain any superdecomposable object, then $\N_{\mathbf{t}}$ has the weak isolation property.
\item If there are no superdecomposable pure-injective objects in $\K{A}$, then every $\N\in\CosiltZg{A}$ has the weak isolation property.
\end{enumerate}
\end{proposition}

\begin{proof}
$(1)$ follows immediately from \cite[Prop.~2.5]{Burke:01}; $(2)$ follows from (1) and Theorem \ref{Thm: cosilting and Zg}(3);

 $(3)$ follows from $(2)$ since $\Prod{\N}$ consists of pure-injective objects and is contained in $\K{A}$ for every $\N\in\CosiltZg{A}$.
\end{proof}

\begin{remark}
Every $\N\in\CosiltZg{A}$ has the weak isolation property for $A$ tame hereditary and for $A$ derived discrete.

Indeed, if $A$ is tame hereditary, then the Krull-Gabriel dimension of $A$ is 2 \cite{Geigle:85}. It follows from \cite{Burke:01} and \cite[Thm.~5.5]{Popescu:73} that $\Mod{A}$ does not contain any superdecomposable pure-injective modules.  Since $A$ is hereditary, every object in $\Der{A}$ is a direct sum of stalk complexes (see, for example, \cite[Sec.~1.6]{Krause:07}) and so it follows that also $\Der{A}$ does not contain any superdecomposable pure-injective objects.  In particular, the conditions of Proposition \ref{prop: superdecomp suff cond tri}(3) hold for $A$.

If $A$ is derived discrete, then the claim follows by combining the main result of \cite{BobinskiKrause:15} with \cite{Burke:01} and \cite[Thm.~5.5]{Popescu:73}.
\end{remark}

\subsection{The weak isolation property in the module category}
In general, much more is known about the existence of superdecomposable pure-injective modules over specific classes of artinian rings, than about the existence of superdecomposable pure-injective objects in derived categories.  We finish this section by proving that when there are no superdecomposable pure-injective $A$-modules, every maximal rigid set has the weak isolation property.  

Let $\tpair{t}{f}\in\tors{A}$ and consider $(\Z_{\mathbf{t}}, \I_{\mathbf{t}}) \in \Cosiltpair{A}$.  We say that $M\in \Z_{\mathbf{t}}$ (respectively $I\in\I_{\mathbf{t}}$) is \textbf{mutable} in $(\Z_{\mathbf{t}}, \I_{\mathbf{t}})$ if $\mu_M$ (respectively I[-1]) is mutable in $\N_{\mathbf{t}}$. Recall from Corollary \ref{cor: mutable means exactly two completions} that we can identify mutable elements of $\N_{\mathbf{t}}$ as those for which there is a unique way to replace them and obtain a new maximal rigid set. We can translate this directly into a statement about cosilting pairs. 

\begin{corollary} \label{cor: mutable means exactly two completions pairs}
Let $(\Z,\I) \in\Cosiltpair{A}$ and $M\in\Z$ (respectively $M\in \I$).  Then $M$ is mutable if and only if there exists a unique $(\Z', \I') \in\Cosiltpair{A}$ such that $(\Z',\I')\neq (\Z,\I)$ and $\Z \setminus \{M\} \subseteq \Z'$ and $\I \subseteq \I'$ (respectively $\Z  \subseteq \Z'$ and $\I \setminus \{M\}\subseteq \I'$).
\end{corollary}

We wish to relate mutability to isolation in the induced Ziegler topology on $\Z$. To simplify matters, we will assume that $A$ is a finite-dimensional algebra.  Recall from Theorem \ref{Thm: I is special fd algebra} that, in this case, all $I\in\I$ are mutable because $I[-1]$ is an $\N_{\mathbf{t}}$-injective envelope of a simple of the form $S=T[-1]$, and such simple objects are always finitely presented by Proposition~\ref{simples in the heart}. In the next theorem, we give a condition for the mutability of $M\in\Z_{\mathbf{t}}$.

Following the notation of \cite[App.~A]{ALS1}, we let $\G_A$ denote the locally coherent Grothendieck category consisting of additive covariant functors from $\mod{A^{op}}$ to the category of abelian groups.

\begin{lemma}\label{isolation quotient}\cite[Prop.~11.17]{KrauseHab}
Let $I$ be an ideal of $A$.  If $\G_A$ has the isolation property, then so does $\G_{A/I}$.
\end{lemma}

\begin{theorem}\label{thm: iso prop}
Suppose that $\G_A$ has the isolation property.  Then the following statements are equivalent for $(\Z, \I)\in\Cosiltpair{A}$ and $M\in \Z$.
\begin{enumerate}
\item There exists a (unique) $(\Z', \I')\in\Cosiltpair{A}$ with $\Z' \neq \Z$, $\I \subseteq \I'$ and $\Z \setminus \{M\} \subseteq \Z'$.
\item $M$ is isolated in $\Z$.
\item $M$ is mutable in $(\Z, \I)$.
\end{enumerate}
\end{theorem}
\begin{proof}
(1)$\Rightarrow$(2) The subset $\Z \cap \Z'$ is closed in $\Z$ by Theorem \ref{Thm: cosilting and Zg}(2) and $\{M\} = \Z \setminus (\Z \cap \Z')$.

(2)$\Rightarrow$(3)  Suppose $\t\in\tors{A}$ such that $(\Z, \I) = (\Z_{\mathbf{t}}, \I_{\mathbf{t}})$ and $t = \Tpair{T}{F}$ the corresponding torsion pair of finite type. We must show that $\mu_M$ is mutable or, equivalently by Corollary \ref{Thm: mutation}, that $\mu_M$ the $\N_{\mathbf{t}}$-injective envelope of a finitely presented simple object in $\heart{\mathbf{t}}$.

As in Subsection \ref{sec: cosilting}, we use the notation $\bar{A} = A/\Ann{\Z}$. By \cite[Thm.~1, Cor.~9]{Prest:96} the natural map $\Zg{\bar{A}} \to \Zg{A}$ induces a homeomorphism onto a closed subset of $\Zg{A}$ and so $M$ is also isolated in $\Z$ considered as a closed subset of $\Zg{\bar{A}}$.  Since $\Z$ is closed in $\Zg{\bar{A}}$, we have that $(\S, \R) := ({}^{\perp_0}(-\otimes_{\bar{A}}\Z), \Cogen{-\otimes_{\bar{A}}\Z})$ is a hereditary torsion pair of finite type in $\G_{\bar{A}}$. The functor $\Mod{\bar{A}} \to \G_{\bar{A}} / \R$ given by $N \mapsto q(-\otimes_{\bar{A}}N)$ where $q$ is the canonical localisation functor induces a homeomorphism $\Z \to \Sp{\G_{\bar{A}} / \R}$ and an additive equivalence $\Prod{\Z} \to \Inj{\G_{\bar{A}} / \R}$.  By Lemma \ref{isolation quotient}, the quotient category $\G_{\bar{A}}/\R$ satisfies the weak isolation property, so $q(-\otimes_{\bar{A}}M)$ is the injective envelope of a finitely presented simple object in $\G_{\bar{A}} / \R$.  

Now, let $(\bar{\T}, \bar{\F}):= ({}^{\perp_0}\Z, \Cogen{\Z})$ be the cotilting torsion pair associated to $(\Z, \I)$ in $\Mod{\bar{A}}$.  By \cite{stovicek}, we have that the additive  equivalence $\Prod{\Z} \to \Inj{\G_{\bar{A}} / \R}$ given by $N \mapsto q(-\otimes_{\bar{A}}N)$ extends to an equivalence of categories between the HRS-tilt $\heart{}$ at $(\bar{\T}, \bar{\F})$ in $\Der{\bar{A}}$ and $\G_{\bar{A}} / \R$. Combining this with the previous paragraph, we obtain that $M$ is the injective envelope of a finitely presented simple object $S$ in the cotilting heart $\heart{}$. 

By \cite[Thm.~A and Thm.~B]{AHL}, there exists one of the following two short exact sequences in $\Mod{\bar{A}}$: \begin{equation} 0 \to F \to M  \overset{f}{\to} \bar{M} \to 0 \:\:\:\:\text{ or }\:\:\:\: 0 \to M \overset{f}{\to} \bar{M} \to T \to 0\end{equation} where, in either case, the morphism $f$ is a left almost split map in $\bar{\F}$ and $F$ (respectively $T$) is torsion-free almost torsion (respectively torsion almost torsion-free) with respect to $(\bar{\T}, \bar{\F})$.  Moreover $M$ is the injective envelope of the simple object $S \cong F$ (respectively $S \cong T[-1]$) in $\heart{}$.  Since $S$ is finitely presented in $\heart{}$, we conclude that the module $F$ (respectively $T$) is finitely presented in $\Mod{\bar{A}}$ by Proposition \ref{simples in the heart}. 

In the first case $F$ is finitely presented torsion-free almost torsion also in $\Mod{A}$ by \cite[Prop.~5.10]{ALS1} (noting that finitely presented $\bar{A}$-modules are also finitely presented when considered as $A$-modules).  Therefore $M$ is the critical that corresponds to $F$ under the bijection in Theorem \ref{Thm: critical tf/t} and hence $\mu_M$ is the $\N_{\mathbf{t}}$-injective envelope of the finitely presented simple $S\cong F$ in $\heart{\mathbf{t}}$ by Corollary \ref{cor: inj env of simples}.  

In the second case $T$ is finitely presented torsion almost torsion-free also in $\Mod{A}$ by Proposition \ref{prop: upgrade t/tf special}.  Therefore $M$ is the special that corresponds to $T$ under the bijection in Theorem \ref{all in Ical appear} and hence $\mu_M$ is the $\N_{\mathbf{t}}$-injective envelope of the finitely presented simple $S\cong T[-1]$ in $\heart{\mathbf{t}}$ by Corollary \ref{cor: inj env of simples}.

(3)$\Rightarrow$(1) This follows from Theorem \ref{bij_posets}(3) and Corollary \ref{cor: mutable means exactly two completions}.
\end{proof}

We have the following corollary, meaning that, whenever $\G_A$ has the isolation property, every element of a maximal rigid subset is either mutable or immutable.

\begin{corollary}\label{global}
If $\G_A$ has the isolation property, then every $\N\in \CosiltZg{A}$ has the weak isolation property. 
\end{corollary}
\begin{proof}
Let $(\Z, \I)$ be the cosilting pair associated to $\N$. By Lemma \ref{lem: weak isol mutability}, we must prove that, if $\mu\in \N$ is isolated, then it is mutable in $\N$. Assume $\mu \in \N$ is isolated. First note that, if $\mu = I[-1]$. then $\mu$ is mutable by Corollary \ref{Thm: mutation}. Now suppose $\mu = \mu_M$ for some $M\in \Z$. By assumption, $\M := \N \setminus \{\mu\}$ is closed and rigid. By Proposition \ref{prop: closed sets}, we have that $\H{}(\M) = \Z \setminus \{M\}$ is closed. In other words, $M$ is isolated in $\Z$. By Theorem \ref{thm: iso prop}, we conclude that $\mu$ is mutable in $\N$.
\end{proof}

Once again we can we can apply \cite[Prop.~2.5]{Burke:01} to obtain a sufficient condition for $\G_A$ to have the isolation property.

\begin{proposition}
If $\Mod{A}$ doesn't have any superdecomposable pure-injective modules, then every $\N\in \CosiltZg{A}$ has the weak isolation property.
\end{proposition}

\begin{remark}
By applying the previous proposition, we obtain that every $\N\in\CosiltZg{A}$ has the weak isolation property for the following finite-dimensional algebras:
\begin{description}
\item[\cite{Auslander}] Finite representation type;
\item[\cite{Geigle:85}] Tame hereditary algebras;
\item[\cite{LPP}] Domestic string algebras;
\item[\cite{JaworskaPastuszakPastuszakBobinski}] Tilted algebras and cluster tilted algebras of tame type;
\item[\cite{Pastuszak}] Domestic standard self-injective algebras;
\item[\cite{Wenderlich}] Strongly simply connected and domestic algebras;
\item[\cite{Skowronski}] Domestic cycle finite algebras.
\end{description}
\end{remark}

\section{Example: cluster-tilted algebras of type $\tilde{A}$}\label{Sec: example}
In this final section, we will consider a family of finite-dimensional algebras over an algebraically closed field $\mathbb{K}$ that arise from triangulations of an annulus. We will give an overview of the cosilting pairs and their mutations for these algebras without giving too many technical details. We refer to \cite{BaurLaking:22} and the references therein for details. 

\subsection{The algebras} We fix an annulus $\mathsf{S}$ with a finite set $\mathsf{M}$ of marked points in the boundary such that both boundary components contain at least one marked point.

For every triangulation $\Gamma = \{\gamma_1, \dots, \gamma_n\}$ of $(\mathsf{S}, \mathsf{M})$, one can define an algebra $A(\Gamma)$ to be the bound path algebra of a quiver with relations determined by the triangulation. The arcs in the triangulation give the vertices of the quiver; the orientation of the surface determines the arrows; and the relations are given by all paths of length two that are contained inside an internal triangle. 

\begin{figure}[h!]
\includegraphics[width=.6\textwidth]{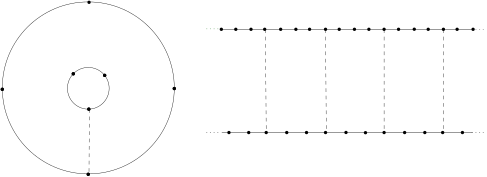}
\includegraphics[width=.8\textwidth]{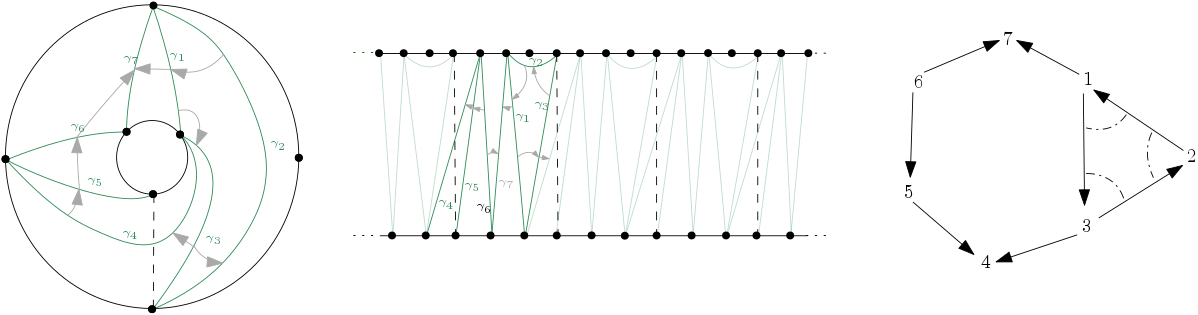}
\caption{Top left: An annulus with 3 marked points on the inner boundary and 4 marked points on the outer boundary. Top right: The universal cover of the same annulus with marked points. Bottom left: Example of a triangulation with arrows of quiver indicated. Bottom middle: The same triangulation in the universal cover. Bottom right: Quiver with relations determined by the triangulation.}\label{fig:quiver}
\end{figure} 

\subsection{The Ziegler spectrum $\Zg{A(\Gamma)}$}\label{Subsec: Zg} The indecomposable pure-injective modules over $A(\Gamma)$ were classified in \cite{Puninski}. We refer to \cite{BaurLaking:22} for details on how to construct these modules explicitly. They fall into two families: string modules and band modules. 

For every (finite or asymptotic) arc $\alpha$ in $(\mathsf{S}, \mathsf{M})$ with end-point(s) in $\mathsf{M}$ that is not already contained in $\Gamma$, there exists a string module $M(\alpha)\in\Zg{A(\Gamma)}$. 

\begin{figure}[h!]
\includegraphics[width=.6\textwidth]{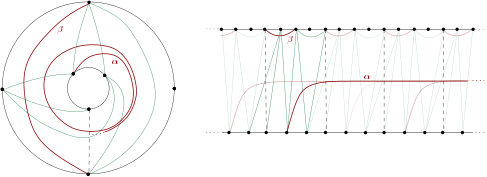}
\caption{Left: Example of an asymptotic are $\alpha$ and a finite arc $\beta$. Right: The same arcs $\alpha$ and $\beta$ occurring in the universal cover.}\label{fig:arcs}
\end{figure} 

The unique closed curve (up to homotopy) represents an infinite family of band modules, all but one of which are indexed by a parameter of the form $(\lambda, n) \in \mathbb{K}^* \times  \mathbb{N} \cup \{\infty, -\infty\}$. We denote the module with parameter $(\lambda, n)$ by $M(\lambda, n)$.  The modules $M(\lambda, \infty)$ are called \textbf{Pr\"ufer modules} and the modules $M(\lambda, -\infty)$ are called \textbf{adic modules}. The remaining band module $G$ is called the \textbf{generic module}. 

\begin{figure}[h!]
\includegraphics[width=.6\textwidth]{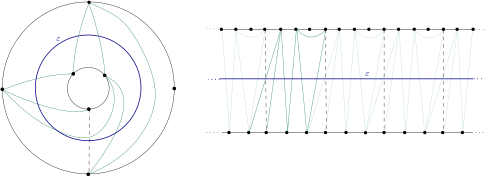}
\caption{The unique closed curve up to homotopy representing the family of band modules.}\label{fig:closed curve}
\end{figure} 

A string module $M(\alpha)$ is finite-dimensional if and only if $\alpha$ is a finite arc. A band module is finite-dimensional if and only if it is of the form $M(\lambda, n)$ with $n\in \mathbb{N}$.

Since the quiver defining $A(\Gamma)$ has only one band, we can apply the results of the paper \cite{PrestPuninski} to describe the topology on $\Zg{A(\Gamma)}$. Firstly we note that the finite-dimensional modules are both closed and isolated points of $\Zg{A(\Gamma)}$.  Let $\U_0$ denote the open set consisting of all finite-dimensional modules. Then the modules $M(\alpha)$ for $\alpha$ asymptotic, and the modules $M(\lambda, \infty)$ and $M(\lambda, -\infty)$ for $\lambda\in \mathbb{K}^*$ are isolated in the subspace topology on $\Zg{A(\Gamma)}\setminus \U_0$. If we let $\U_1$ be the set of all these infinite-dimensional modules, then the set $\U_0\cup \U_1$ is open and the closed set $\Zg{A(\Gamma)}\setminus \U_0\cup \U_1$ contains only $G$.

\subsection{The cosilting pairs}
In \cite[Thm.~4.14]{ALS1} we showed that the cosilting pairs over $A(\Gamma)$ are in bijection with the cosilting modules up to equivalence. Since the cosilting modules over $A(\Gamma)$ were classified up to equivalence in \cite{BaurLaking:22}, we can use this to describe the elements of $\Cosiltpair{A(\Gamma)}$.

A vital notion is when two arcs are considered to cross. We do not give a formal definition of this (but the intuitive one is correct) and we refer to the paper \cite{BaurDupont} for full explanation; they call pairwise non-crossing arcs \emph{compatible}.

The data indexing the cosilting $A(\Gamma)$-modules is a tuple $(\mathcal{C}, \P, \mathcal{A}, \ast)$ where \begin{itemize}
\item $\mathcal{C}$ is a set of non-crossing arcs that may contain asymptotic and/or finite arcs, including those contained in $\Gamma$. Arcs that share end-points are not considered to be crossing;
\item $\P$ and $\mathcal{A}$ are subsets of $\mathbb{K}^*$ such that $\mathcal{A} \cap \P = \emptyset$;
\item The symbol $\ast$ either represents $G$ or $\cancel{G}$.\end{itemize} 
The tuple is subject to the following rules. \begin{description}
\item[(C1)] The set $\mathcal{C}$ is maximal among all sets of non-crossing arcs.
\item[(C2)] If $\mathcal{C}$ contains an arc that crosses $\varepsilon$, then $\mathcal{A} = \P = \emptyset$ and $\ast = \cancel{G}$.
\item[(C3)] If $\mathcal{C}$ does not contain an arc that crosses $\varepsilon$, then $\mathcal{A} \cup \P = \mathbb{K}^*$ and $\ast= G$.
\end{description}

\begin{remark}\label{rmk: BD results}
The sets $\C$ that are maximal among all sets of non-crossing arcs are studied in the paper \cite{BaurDupont}. It is shown that the number of arcs in every set $\C$ is equal to the cardinality of $\mathsf{M}$. It is also shown that $\C$ contains an arc that crosses $\varepsilon$ if and only if $\C$ contains no asymptotic arcs. They also show that, if $\C$ contains an arc that crosses $\varepsilon$, then $\C$ contains at least two arcs that cross $\varepsilon$ and that, if $\C$ contains an asymptotic arc, then $\C$ contains at least two asymptotic arcs.
\end{remark}

\begin{figure}[h!]
\includegraphics[width=.6\textwidth]{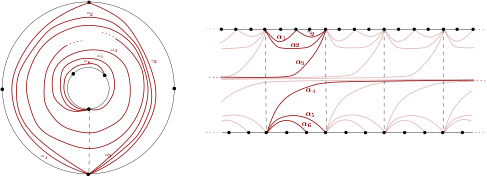}
\caption{An example of a set $\C$ that satisfies the conditions with respect to $\Gamma$ defined in Figure \ref{fig:quiver} but does not containing any arc that crosses $\varepsilon$. We note that $\gamma_2$ is contained in $\Gamma$.}\label{fig:asymp triangulation}
\end{figure}

\begin{figure}[h!]
\includegraphics[width=.6\textwidth]{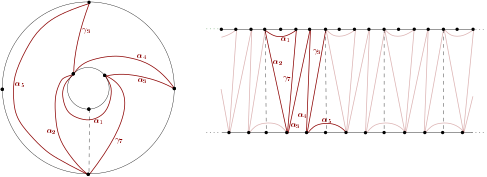}
\caption{An example of a set $\C$ that satisfies the conditions with respect to $\Gamma$ defined in Figure \ref{fig:quiver} and contains an arc that crosses $\varepsilon$. We note that $\gamma_3$ and $\gamma_7$ are contained in $\Gamma$.}\label{fig:finite triangulation}
\end{figure} 

\newpage

\begin{proposition}
The cosilting pairs over $A(\Gamma)$ are indexed by the tuples $(\mathcal{C}, \P, \mathcal{A}, \ast)$ and, given such a tuple, the corresponding cosilting pair $(\Z, \I)$ is given by the following.
\begin{itemize}
\item $\Z = \{ M(\alpha) \mid \alpha \in \mathcal{C}\setminus \Gamma \}\cup \{M(\lambda, \infty) \mid \lambda \in \P\} \cup \{M(\lambda, -\infty) \mid \lambda \in \mathcal{A}\} \cup \{G \mid \text{ if } \ast= G\}$.
\item $\I = \{ I(i) \mid \alpha = \gamma_i \in \mathcal{C} \cap \Gamma\}$.
\end{itemize}
\end{proposition}
\begin{proof}
It was shown in \cite[Thm.~2.36]{BaurLaking:22} that the cosilting modules (up to equivalence) are indexed by the tuples $(\mathcal{C}, \P, \mathcal{A}, \ast)$. Given a tuple $(\mathcal{C}, \P, \mathcal{A}, \ast)$, the corresponding cosilting module is shown to be $C:=\prod_{N \in \Z} N$ where $\Z$ is defined as in the statement of the proposition.  

We confirm that, if $C$ is the cosilting module corresponding the the tuple $(\mathcal{C}, \P, \mathcal{A}, \ast)$, then $(\Z, \I)$ is the cosilting pair determined by $C$ in the bijection given in Theorem \ref{bij_cosilt mod cpx}. According to Theorem \ref{bij_cosilt mod cpx}, we must check that $\Z = \Ind{\Prod{C}}$ and that $\I = \InjInd{A(\Gamma)} \cap C^{\perp_0}$.

To show that $\Z = \Ind{\Prod{C}}$, it suffices to show that $\Z$ is a closed subset of the Ziegler spectrum, see Remark~\ref{closedisproductrigid}.

First we assume that $\C$ contains an arc that crosses $\varepsilon$. By Remark \ref{rmk: BD results}, $\C$ contains only finitely many finite arcs, $\A = \P = \emptyset$ and $\ast= \cancel{G}$. Thus the set $\Z$ is a finite set of closed points and is therefore closed. Next we suppose that $\C$ does not contain any arc that crosses $\varepsilon$ so that $\A \cup \P = \mathbb{K}^*$ and $\ast= G$. By the description of the topology given in Subsection \ref{Subsec: Zg}, we have that the set $\Z\setminus \U_0$ is obtained from the closed set $\Zg{A(\Gamma)}\setminus \U_0$ by removing isolated points and so is closed. Moreover the set $\Z \cap \U_0$ consists of finitely many closed points. Thus $\Z$ is a finite union of closed sets and hence is closed.

Finally, we show that $\I = \InjInd{A(\Gamma)} \cap C^{\perp_0}$. Consider an indecomposable injective $A(\Gamma)$-module $I(i)$ with simple socle $S(i)$ corresponding to the vertex of the quiver determined by the arc $\gamma_i\in \Gamma$. Then a standard argument yields that $\Hom{A(\Gamma)}(C, I(i))$ is non-zero if and only if the simple module $S(i)$ occurs as a subquotient of $C$. This happens exactly when $C$ is supported at vertex $i$ or, in other words, when there is an arc in $\C$ that crosses $\gamma_i$. It follows that $\Hom{A(\Gamma)}(C, I(i))=0$ if and only if $\gamma_i \in \C$.
\end{proof}

\begin{example}
Consider the tuple $(\mathcal{C}, \mathbb{K}^*, \emptyset, G)$ where $\mathcal{C}$ is the set given in Figure \ref{fig:asymp triangulation}. Then the corresponding cosilting pair is $(\Z, \I)$ where $\Z = \{M(\alpha_i) \mid i= 1,\dots, 6\} \cup \{M(\lambda, \infty) \mid \lambda \in \mathbb{K}^*\} \cup \{G\}$ and $\I = \{I(2)\}$. 
\end{example}

\begin{example}
Consider the tuple $(\mathcal{C}, \emptyset, \emptyset, \cancel{G})$ where $\mathcal{C}$ is the set given in Figure \ref{fig:finite triangulation}. Then the corresponding cosilting pair is $(\Z, \I)$ where $\Z = \{M(\alpha_i) \mid i= 1,\dots, 5\}$ and $\I = \{I(3), I(7)\}$. 
\end{example}

 Using the bijection in Theorem \ref{bij_posets}, we also obtain a description of the maximal rigid sets over $A(\Gamma)$.
 
 \begin{corollary}
 The maximal rigid sets over $A(\Gamma)$ are indexed by the tuples $(\mathcal{C}, \P, \mathcal{A}, \ast)$ and, given such a tuple, the corresponding maximal rigid set is 
 	\begin{align*} \{\mu_{M(\alpha)} \mid\alpha \in \mathcal{C}\setminus \Gamma \} \cup  \{\mu_{M(\lambda, \infty)} \mid \lambda \in \P\}\cup \{\mu_{M(\lambda, -\infty)} \mid & \lambda \in \mathcal{A}\} \cup \\ &\{\mu_G \mid \text{ if } \ast= G\} \cup \{ I(i)[-1] \mid \alpha = \gamma_i \in \mathcal{C} \cap \Gamma\}. \end{align*}
 \end{corollary} 

\subsection{Irreducible mutations} 

Let $(\Z,\I)$ be a cosilting pair over $A(\Gamma)$ indexed by the tuple $(\C, \P, \A, \ast)$. In Corollary \ref{cor: mutable means exactly two completions pairs} we saw that a module $M$ in $\Z$ or $\I$ is mutable if and only if there is a unique way to replace $M$ to obtain a new cosilting pair. Since we have a full classification of cosilting pairs, we can find all of the mutable elements by observation. 

A key element to being able to identify mutations is to understand how we can combinatorially `mutate' the tuples $(\C, \P, \A, \ast)$. In the case where $\C$ contains an arc that crosses $\varepsilon$, we have that $\P$ and $\A$ must satisfy the conditions that $\P\cap\A= \emptyset$ and $\P\cup\A= \mathbb{K}^*$. It is clear that we can move a parameter from $\P$ to $\A$ or vice versa and this will yield another tuple that satisfies the required conditions. It's also clear that it is impossible to change the value of $\ast$ from $G$ to $\cancel{G}$ and add only one element to $\C, \A$ or $\P$ without violating the conditions (C1) - (C3).

It remains to consider when it is possible to `mutate' the elements of $\C$. This is the main topic of the paper \cite{BaurDupont} in which they show that, for any $\alpha\in\C$, there is a flip operation that allows one to replace $\alpha$ with an arc $\beta\notin\C$ such that the resulting set is still a maximal set of non-crossing arcs. By Remark \ref{rmk: BD results}, the set $\C$ contains an arc that crosses $\varepsilon$ if and only if the mutation of $\C$ contains an arc that crosses $\varepsilon$.

Combining these combinatorial observations with Corollary \ref{cor: mutable means exactly two completions pairs}, we can describe all the possible irreducible mutations. We summarise these results in the following proposition.

\begin{proposition}\label{prop: annulus mutable points}
Let $(\Z,\I)$ be a cosilting pair over $A(\Gamma)$ indexed by the tuple $(\C, \P, \A, \ast)$. Then the following statements hold.
\begin{enumerate}
\item If $\C$ contains an arc that crosses $\varepsilon$, then every element of $(\Z, \I)$ is mutable and the resulting cosilting pair will correspond to a tuple that contains an arc that crosses $\varepsilon$.
\item If $\C$ does not contain an arc that crosses $\varepsilon$, then every element of $(\Z, \I)$ except $G$ is mutable and the resulting cosilting pair will correspond to a tuple that does not contain an arc that crosses $\varepsilon$. 
\item We have exactly the following possible ways to mutate $(\Z, \I)$ to obtain some $(\Z', \I')$.
\begin{enumerate}
\item For $\alpha \in \C \setminus \Gamma$, let $\beta \notin \C$ be the flip of $\alpha$. If $\beta \notin \Gamma$, then $\Z' = (\Z \setminus \{M(\alpha)\}) \cup \{M(\beta)\}$ and $\I = \I'$. If $\beta = \gamma_j \in \Gamma$, then $\Z'=\Z\setminus \{M(\alpha)\}$ and $\I' = \I \cup \{I(j)\}$.

\item For $\gamma_i\in\C\cap \Gamma$, let $\beta$ be the flip of $\gamma_i$. If $\beta \notin \Gamma$, then $\Z' = \Z \cup \{M(\beta)\}$ and $\I' = \I\setminus \{I(i)\}$. If $\beta = \gamma_j \in \Gamma$, then $\Z'=\Z$ and $\I' = (\I \setminus \{I(i)\}) \cup \{I(j)\}$.
\item For every $\lambda\in\P$ (respectively $\lambda\in\A$), the module $M(\lambda, \infty)\in\Z$ (respectively the module $M(\lambda, -\infty)\in\Z$) is replaced by $M(\lambda, -\infty)\in \Z'$ (respectively by $M(\lambda, \infty)\in \Z'$).
\end{enumerate}
\end{enumerate}
\end{proposition}

\begin{figure}[h!]
\includegraphics[width=.6\textwidth]{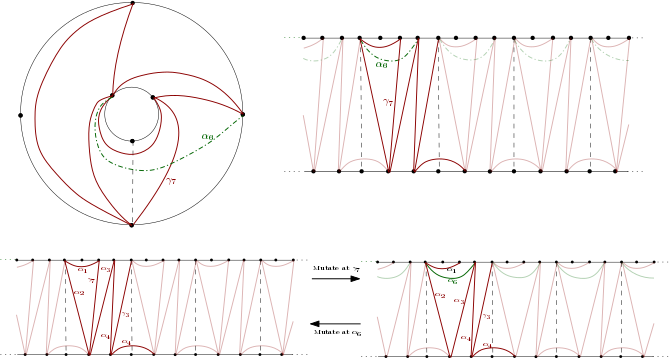}
\caption{An example of the flip operation on a set $\C$ that contains an arc crossing $\varepsilon$. The arc $\alpha_6$ is the unique way to replace the arc $\gamma_7$ to obtain a new maximal collection of non-crossing arcs.}\label{fig:finite mutation}
\end{figure} 

\begin{example}
Consider the tuple $(\mathcal{C}, \emptyset, \emptyset, \cancel{G})$ where $\mathcal{C}$ is the set given by the left hand image in Figure \ref{fig:finite mutation}. Then the corresponding cosilting pair is $(\Z, \I)$ where $\Z = \{M(\alpha_i) \mid i= 1,\dots, 5\}$ and $\I = \{I(3), I(7)\}$. We can then mutate at the module $I(7)$ and we obtain the mutated cosilting pair $(\Z', \I')$ where $\Z' = \{M(\alpha_i) \mid i= 1,\dots, 6\}$ and $\I' = \{I(3)\}$.
\end{example}

\begin{figure}[h!]
\includegraphics[width=.6\textwidth]{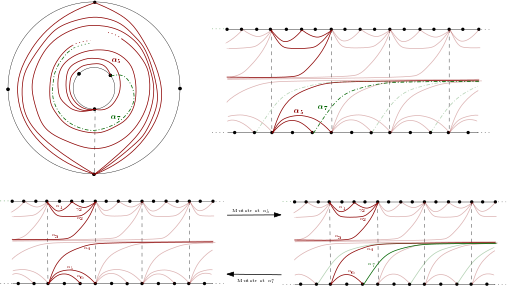}
\caption{An example of the flip operation on a set $\C$ that does not contain an arc crossing $\varepsilon$. The arc $\alpha_7$ is the unique way to replace the arc $\alpha_5$ to obtain a new maximal collection of non-crossing arcs.}\label{fig:asymp mutation}
\end{figure}

\begin{example}
Consider the tuple $(\mathcal{C}, \mathbb{K}^*, \emptyset, G)$ where $\mathcal{C}$ is the set given in the left hand image in Figure \ref{fig:asymp mutation}. Then the corresponding cosilting pair is $(\Z, \I)$ where $\Z = \{M(\alpha_i) \mid i= 1,\dots, 6\} \cup \{M(\lambda, \infty) \mid \lambda \in \mathbb{K}^*\} \cup \{G\}$ and $\I = \{I(2)\}$. We can then mutate at the module $M(\alpha_5)$  and we obtain the mutated cosilting pair $(\Z', \I')$ where \[\Z' = \{M(\alpha_i) \mid i= 1,\dots, 4, 6, 7\} \cup \{M(\lambda, \infty) \mid \lambda \in \mathbb{K}^*\} \cup \{G\} \text{ and } \I' = \{I(2)\}.\]
\end{example}

\begin{remark}
From our description of the irreducible mutations of cosilting pairs over $A(\Gamma)$ we can deduce the underlying graph of the Hasse quiver of $\tors{A}$.  The part of the Hasse quiver of $\tors{A(\Gamma)}$ that corresponds to tuples where $\C$ contains an arc that crosses $\varepsilon$ (or equivalently to functorially finite torsion pairs) has a 3-regular underlying graph isomorphic to the dual of the arc complex described in \cite[Sec.~3]{FominShapiroThurston}. The remainder of the underlying graph of the Hasse quiver of $\tors{A(\Gamma)}$ is isomorphic to the product of the Hasse diagram of the power set $\mathcal{P}(\mathbb{K}^*)$ with the boundary of the asymptotic exchange graph described in \cite{BaurDupont}
\end{remark}

\bibliographystyle{plain}
\bibliography{refs}

\end{document}